\documentclass[a4paper,10pt]{article}

\usepackage[utf8x]{inputenc}
\usepackage[english]{babel}

\usepackage{amsmath}
\usepackage{graphicx}
\usepackage{hyperref}

\usepackage{amsfonts}
\usepackage{amsthm}
\usepackage{amssymb}
\usepackage{cite} 

\newcommand{\R}{\mathbb{R}} 
\newcommand{\N}{\mathbb{N}}

\newcommand{\graph}{\operatorname{graph}}

\newtheorem{theorem}{Theorem}[section]
\newtheorem{lemma}[theorem]{Lemma}

\theoremstyle{definition}
\newtheorem{definition}[theorem]{Definition}
\newtheorem{remark}[theorem]{Remark}

\numberwithin{equation}{section}

\title{The Helfrich Boundary Value Problem}

\author{Sascha Eichmann\thanks{The author thanks Prof. Reiner Sch\"atzle for discussing the Helfrich energy and providing insight into geometric measure theory.}\\
Mathematisch-Naturwissenschaftliche Fakultät,\\
Eberhard Karls Universität Tübingen,\\
Auf der Morgenstelle 10,\\
D-72076 Tübingen, Germany\\
E-mail: \href{mailto:sascha.eichmann@math.uni-tuebingen.de}{sascha.eichmann@math.uni-tuebingen.de}\\
Phone: +49/7071/2976886}

\begin{document}
 \maketitle

\begin{abstract}
We construct a branched Helfrich immersion satisfying Dirichlet boundary conditions. The number of branch points is finite.\\
We proceed by a variational argument and hence examine the Helfrich energy for oriented varifolds.
The main contribution of this paper is a lower semicontinuity result with respect to oriented varifold convergence for the Helfrich energy and a minimising sequence.
For arbitrary sequences this is false by a counterexample of Gro\ss e-Brauckmann.

\end{abstract}

\section{Introduction}
\label{sec:1}
The Helfrich energy (or Canham-Helfrich energy) was introduced by Helfrich in \cite{Helfrich} resp. Canham in \cite{Canham}. 
We will deal with the following variant, which is defined for an oriented immersion $f:\Sigma\rightarrow \R^3$ with a smooth $2$-dimensional manifold $\Sigma$ 
\begin{equation}
\label{eq:1_1}
 W_{H_0,\lambda}(f) := \int_\Sigma (\bar{H}-H_0)^2 + \lambda d\mu_g.
\end{equation}
Here $\mu_g$ denotes the area measure induced by $f$ and $\bar{H}$ the mean curvature of $f$ with respect to the normal $\nu$ of $f$ induced by the orientation (i.e. the sum of the principal curvatures).
In this article $H_0\in\R$ is arbitrary and $\lambda>0$.
In terms of the mean curvature vector $H$ we can write $\bar{H}=H\cdot\nu$. 
The corresponding Euler-Lagrange equation is as follows (see \cite[Eq. (31)]{OuYang_Helfrich})
\begin{equation}
 \label{eq:1_2}
 2\Delta_f \bar{H} + 4\bar{H}\left(\frac{1}{4}\bar{H}^2-K\right)-2H_0 K -H_0^2\bar{H}-\lambda \bar{H}=0.
\end{equation}
Here $\Delta_f$ denotes the Laplace-Beltrami and $K$ the Gauss curvature of $f$.
We call an immersion satisfying this equation a Helfrich immersion or just Helfrich. 

Apart from geometrical interest, the Helfrich energy has applications in e.g. biology in modeling red blood cells or lipid bilayers (see e.g. \cite{Canham}, \cite{Helfrich} or \cite{OuYang}).
In mathematics itself research has been mainly focused on the Willmore energy (i.e. $H_0=0$, $\lambda=0$), 
which was revived by Willmore in \cite{Willmore}, but has already been studied by e.g. Thomsen in \cite{Thomsen}. 
A general discussion of such curvature integrals for variational purposes can be found in \cite{Nitsche}.
In context of the Willmore energy mainly compact surfaces without boundary were considered, see e.g. \cite{Simon}, \cite{BauerKuwert} or \cite{NevesMarques}.
Compact surfaces are also considered for minimising the Helfrich energy under fixed area and enclosed volume. 
In this regard some results have been established for axisymmetric surfaces, see e.g. \cite{ChoksiVeroni} and \cite{ChoksiMorandottiVeneroni}.
For the Willmore energy the scaling invariance simplifies this problem to prescribing the Isoperimetric Ratio and was solved for genus $0$ surfaces in \cite{Schygulla}.

We on the other hand will focus on a Dirichlet boundary value problem for Helfrich immersions. 
For the Willmore energy existence results have already been achieved by e.g. Schätzle in the class of branched immersions in \cite{Schaetzle}. 
In \cite{DaLioPalmurellaRiviere} these existence results were improved by different methods and e.g. branch points were excluded from the boundary.
Our reasoning will be heavily inspired by Sch\"atzle's article \cite{Schaetzle} and Simon's work \cite{Simon}.
In the class of surfaces of revolution existence results for Willmore surfaces were obtained in the series of the papers \cite{DallDeckGru}, \cite{DallFroehGruSchie} and \cite{EichmannGrunau}.
These results even have some extensions to the Helfrich energy in \cite{Scholtes} or \cite{Doemeland}.
In the class of graphs existence of a very weak minimiser in the class of functions of bounded variation for the Helfrich and Willmore energy was obtained in \cite{DeckGruRoe}.

We will show an existence result for Dirichlet boundary data in the class of branched immersions.
The Helfrich equation \eqref{eq:1_2} is of fourth order, where many established techniques like the maximum principle do not apply.
Hence we will use a variational approach with oriented varifolds (see \cite{Hutchinson} or appendix \ref{sec:B}).
Unfortunately the Helfrich energy \eqref{eq:1_1} is in general not lower semicontinuous with respect to varifold convergence (see \cite{GrosseBrauck}).
Delladio was able to overcome this obstacle in \cite{Delladio} (see also \cite{AnzellottiSerapioniTamanini} for combining Gauss graphs and curvature integrals)
by working with current convergence of the associated Gauss graphs of the surfaces.
Unfortunately he had to assume $C^2$-regularity on the limit of a minimising sequence for his argument to work, which is a-priori not clear.\\ 

Let us now describe our Dirichlet problem in detail:
To do this we need a given smooth one dimensional compact and embedded manifold  $\Gamma\subset\R^3$ with a smooth unit normal vectorfield ${\bf n}\in N\Gamma\subset \R^3$.
Then our Dirichlet boundary conditions can be stated for a smooth, two dimensional immersion $f:\Sigma\rightarrow\R^3$:
\begin{equation}
\label{eq:1_3}
\left\{\begin{array}{c}f|_{\partial \Sigma}\rightarrow \Gamma \mbox{ is a diffeomorphism}\\
        {\bf co}_f(f(x))={\bf n}(f(x)),\quad x\in\partial\Sigma,
       \end{array}\right. 
\end{equation}
here ${\bf co}_f$ denotes the inner conormal of $f$.

Our main theorem can be stated as follows:
\begin{theorem}
 \label{1_1}
 Let $\lambda>0$ and $H_0\in\R$, then for any smooth embedded compact oriented one-dimensional manifold $\Gamma\subset \R^3$ with smooth unit normal field $n\in N\Gamma$, 
 there exists a two dimensional branched immersion $f:\Sigma\rightarrow\R^3$ of a compact
 oriented manifold $\Sigma$, which satisfies \eqref{eq:1_3} and is Helfrich outside the finitely many branch points.
 $f$ is smooth outside the branch points and continuous at the branch points.
\end{theorem}
\begin{remark}
 \label{1_2}
Branch points of $f$ may appear on the boundary.
\end{remark}

Let us now summarize the most important parts of the proof und discuss some other possible approaches:
We will describe our minimising sequence as oriented varifolds and apply Hutchinsons compactness result \cite[Thm. 3.1]{Hutchinson} in section \ref{sec:2} to find a limit after extracting a subsequence. 
Then we modify the arguments of \cite[Prop. 2.2]{Schaetzle} resp. \cite[Thm 3.1]{Simon} to show $C^{1,\alpha}\cap W^{2,2}_{loc}$-regularity outside of 
finitely many 'bad' points in Lemma \ref{3_1}. This lemma also shows that the limit can be decomposed into $C^{1,\alpha}\cap W^{2,2}$-graphs near 'good' points.
This enables us to prove lower semicontinuity for our minimising sequence in section \ref{sec:4}, because we will be able to write the Helfrich energy for one such graph variationally,
i.e. in the form of 
\begin{equation*}
\sup \left\{\int h \varphi\, d\mathcal{H}^2|\ \|\varphi\|_{L^2}\leq 1,\ \varphi\mbox{ continuous with compact support}\right\}
\end{equation*}
with some $\mathcal{H}^2$-measurable function $h$ (see \eqref{eq:4_4} for a precise statement). 
For arbitrary sequences of oriented varifolds this does not seem to be possible, since the addition of the two 'square' terms in the general form of the Helfrich energy \eqref{eq:2_1} seems to prevent this.
Furthermore we cannot use Delladio's result \cite[Thm 5.1]{Delladio}, since we do not know if the limit is $C^2$-regular everywhere. 
Adapting Delladio's proof to our setting also does not seem to be possible, since he uses the lower semicontinuity and compactness result of Hutchinson \cite[4.4.2]{Hutchinson} for measure-function pairs. 
Even if we find measure-function pairs and a corresponding integrand to encode the Helfrich energy, it is not clear that the resulting limit measure-function pair is related to any curvature function.
For example such an integrand may look like $F(y,q)=(q-H_0)^2$ with a corresponding measure-function pair $(\mu, h)$ with $h(y)=\bar{H}$. 
For Gro\ss e-Brauckmann's counterexample \cite[p. 550, Remark (ii)]{GrosseBrauck} this would result in $h=1$ and $W_{H_0,0}=0$ for the whole sequence and the limit measure function pair would be of the form $(\mu,y\mapsto 1)$.
This is in contrast to the fact, that in this example the sequence convergences to a plane of multiplicity $2$.\\
In section \ref{sec:5} we collect the needed arguments of \cite{Simon}, \cite{Schaetzle} and \cite{NdiayeSchaetzle} to finish the proof of Theorem \ref{1_1} with Lemma \ref{5_3}.
There we only sketch the arguments, since they have essentially been given in the aforementioned papers.

\section{Compactness}
\label{sec:2}
As in \cite{Schaetzle} we will prove Theorem \ref{1_1} with the direct method of the calculus of variations employing geometric measure theory.
Since the Helfrich energy depends on a given orientation, we will formulate the compactness arguments in the context of oriented varifolds. 
These were introduced by Hutchinson (see \cite[chapter 3]{Hutchinson}). 
For the readers convenience the basic definitions and notations of these objects are summarized in Appendix \ref{sec:B}.

Since we will work with $2$-dimensional varifolds in $\R^3$, we can identify the oriented Grassmannian $G^0(2,3)$ with $\partial B_1(0)=\{\nu\in\R^3:|\nu|=1\}$ by the Hodge star 
operator $*$ (see e.g. \cite[exercise 16-18]{LeeIntroManifold}). 
In our case $*$ is given for $\tau_1\wedge \tau_2\in G^0(2,3)$ as the cross product, i.e. $*(\tau_1\wedge \tau_2)=\tau_1\times\tau_2$.\\
Let $V^0=V^0(M,\theta_\pm,\xi)\in RV^0(\R^3)$ such that, the mean curvature vector $H_{V^0}$ of $\mu_{V^0}$ satisfies $H_{V^0}\in L^2(\mu_{V^0})$ (see \eqref{eq:B_4} for a precise definition). 
Then the Helfrich energy is defined as
\begin{align}
\begin{split}
\label{eq:2_1}
W_{H_0,\lambda}(V^0)=&\int_{G^0(\R^3)} (H_{V^0}(x)-(*\xi)H_0)^2\, dV^0(x,\xi)+\lambda \mu_{V^0}(\R^3)\\
=& \int_M((H_{V^0}-(*\xi(x))H_0)^2\theta_+(x) \\
& + (H_{V^0}+(*\xi(x))H_0)^2\theta_-(x))\, d\mathcal{H}^2(x) + \lambda\int_M\theta_++\theta_-\, d\mathcal{H}^2.
\end{split}
\end{align}
Please note, that by Brakke's orthogonality result (see \cite[§5]{Brakke}) $H_{V_0}(x)$ is parallel to $*\xi(x)$ for $\mu_{V^0}$ a.e. $x$, if $V^0$ is integral.\\
If $M$ is a $2$-dimensional orientable $C^1$-immersion $f:\Sigma\rightarrow \R^3$ 
with a given $\mathcal{H}^2$-measurable orientation $\xi_f:f(\Sigma)\rightarrow G^0(3,2)$, the corresponding oriented integral varifold is
\begin{equation*}
 V^0_f:=V^0(f(\Sigma),\theta_+,\theta_-,\xi).
\end{equation*}
To define the densities let us denote the choosen orientation of $T_x\Sigma$ by $\tau(x)$. 
Then $\theta_+,\theta_-:f(\Sigma)\rightarrow \N_0$ are defined by
\begin{align}
\begin{split}
\label{eq:2_1_1}
 \theta_+(y)&=\sum_{x\in f^{-1}(y)}\operatorname{sign}_+(df(\tau(x))\diagup (\xi_f(y)),\\ \theta_-(y)&=\sum_{x\in f^{-1}(y)}\operatorname{sign}_-(df(\tau(x))\diagup (\xi_f(y)).
\end{split}
 \end{align}
Here 
\begin{equation*}
 \operatorname{sign}_+(df(\tau(x))\diagup (\xi_f(y))=\left\{\begin{array}{cc}1,&\mbox{ if }\xi_f(y)\mbox{ is the same orientation as }df(\tau(x))\\ 0,&\mbox{ else.}\end{array}\right. 
\end{equation*}
Analogously $\operatorname{sign}_-(df(\tau(x))\diagup (\xi_f(y))=1$ if $\xi_f(y)$ is the opposite orientation of $df(\tau(x))$. 
Please note, that these densities are only well defined $\mathcal{H}^2\lfloor f(\Sigma)$ almost everywhere, which is enough to obtain a well defined oriented varifold (see also Figure \ref{fig_1}).
\begin{figure}[h] 
\centering 
\includegraphics{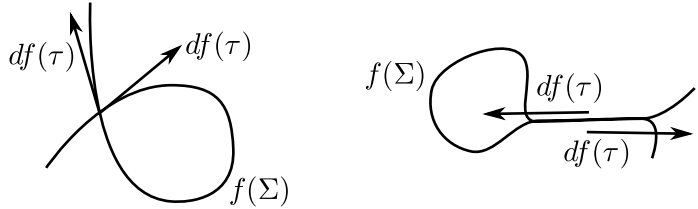}  
\caption{Examples of the behaviour of $df(\tau(x))$}
\label{fig_1}
\end{figure}

Let us now define our minimising sequence: Let $\Sigma$ be a compact orientable connected manifold with boundary. 
Then we consider the sequence of oriented smooth immersions $f_{m}:\Sigma\rightarrow\R^3$ satisfiying
\begin{equation}
 \label{eq:2_2}
\left\{\begin{array}{c} f_m|_{\partial \Sigma}\rightarrow \Gamma \mbox{ is a diffeomorphism}\\
        {\bf co}_{f_m}(f_m(x))={\bf n}(f_m(x)),\quad x\in\partial\Sigma\\
        W_{H_0,\lambda}(f_m)\rightarrow \inf_{f:\Sigma\rightarrow\R^3 \mbox{ satisfies \eqref{eq:1_3}, smooth}}W_{H_0,\lambda}(f)
       \end{array}\right.
\end{equation}
We call the corresponding integral varifold $V^0_{f_m}$ as defined above.

Now we consider a fixed oriented smooth immersion $f_0:\Sigma_0\rightarrow\R^3$ such that
\begin{equation*}
 \left\{\begin{array}{c}f_0|_{\partial \Sigma_0}\rightarrow \Gamma \mbox{ is a diffeomorphism}\\
        {\bf co}_{f_0}(f_0(x))=-{\bf n}(f_0(x)),\quad x\in\partial\Sigma_0,
       \end{array}\right.
\end{equation*}
After an appropriate glueing we obtain $f_{m,0}:=f_m\oplus f_0:\Sigma_\oplus:=\Sigma\oplus\Sigma_0\rightarrow\R^3$ to be a compact $C^{1,1}$ immersion without boundary.
After slightly smoothing around $\Gamma$ we can even assume $f_{m,0}\in C^2$, while still keeping \eqref{eq:2_2}. Hence
\begin{equation}
 \label{eq:2_3}
 W_{H_0,\lambda}(V^0_{f_{m,0}})\leq E:=E(\Gamma,{\bf n}, \Sigma,\lambda,H_0).
\end{equation}
For the sake of brevitiy let us denote 
\begin{equation*}
\mu_m:=\mu_{V^0_{f_m}}=(x\mapsto \mathcal{H}^0(f_m^{-1}(x)))\mathcal{H}^2\lfloor f_m(\Sigma)
\end{equation*}
and 
\begin{equation*}
\mu_0:=\mu_{V^0_{f_0}}=(x\mapsto \mathcal{H}^0(f_0^{-1}(x)))\mathcal{H}^2\lfloor f_0(\Sigma_0).
\end{equation*}
Therefore the mass of $V^0_{f_{m,0}}$ satisfies
\begin{equation*}
 \mu_m+\mu_0 = \mu_{V^0_{f_{m,0}}}.
\end{equation*}
Let us further denote the second fundamental form of $f_m$, $f_0$ and $f_{m,0}$ by $A_m$, $A_0$ and $A_{m,0}$ respectively.
Next we will apply Hutchinson's compactness Theorem \ref{B_2} (see \cite[Theorem 3.1]{Hutchinson}) to $V^0_{f_{m,0}}$. 
Luckily the boundary of the associated current satisfies $\partial[|V^0_{f_{m,0}}|]=0$, since $f_{m,0}$ does not have a boundary either. Let $\xi_{f_{m,0}}$ be the orientation of $f_{m,0}(\Sigma_\oplus)$ and $\theta_{\pm,m}$ the densities as defined above. 
Hence we only need to estimate the mass and the first variation. We use the Cauchy-Schwartz and $\varepsilon$-Young inequality to deduce
\begin{align*}
 & \int_{f_{m,0}(\Sigma_\oplus)}(H_{V^0_{f_{m,0}}}-H_0(*\xi_{f_{m,0}}(x)))^2\theta_{+,m}(x)\, d\mathcal{H}^2(x)+\lambda \int_{(f_{m,0}(\Sigma_\oplus))}\theta_{+,m}\, d\mathcal{H}^2\\
 =&\int_{f_{m,0}(\Sigma_\oplus)}( (H_{V^0_{f_{m,0}}})^2- 2H_0 H_{V^0_{f_{m,0}}}\cdot (*\xi_{f_{m,0}}) + H_0^2 + \lambda)\theta_{+,m}\,d\mathcal{H}^2\\
 \geq &\int_{f_{m,0}(\Sigma_\oplus)}((H_{V^0_{f_{m,0}}})^2- 2|H_0||(H_{V^0_{f_{m,0}}})| + H_0^2+ \lambda)\theta_{+,m}\,d\mathcal{H}^2\\
 \geq & \int_{f_{m,0}(\Sigma_\oplus)}((H_{V^0_{f_{m,0}}})^2- \varepsilon (H_{V^0_{f_{m,0}}})^2 - \frac{1}{\varepsilon}H_0^2 + H_0^2 + \lambda)\theta_{+,m}\,d\mathcal{H}^2\\
 =& \int_{f_{m,0}(\Sigma_\oplus)}((1-\varepsilon)(H_{V^0_{f_{m,0}}})^2 + (\lambda+(1-\frac{1}{\varepsilon})H_0^2)\theta_{+,m})\, d\mathcal{H}^2.
\end{align*}
Since $\lambda>0$ we can choose $\varepsilon>0$ to satisfy $1>\varepsilon$ and $\lambda+(1-\frac{1}{\varepsilon})H_0^2=\frac{\lambda}{2}$. 
Hence we find a constant $C=C(H_0,\lambda)>0$ such that
\begin{equation}
 \label{eq:2_4}
 W_{H_0,\lambda}(V^0_{f_{m,0}})\geq C\int_{f_{m,0}(\Sigma_\oplus)}((H_{V^0_{f_{m,0}}})^2+1)(\theta_{+,m}+\theta_{-,m})\, d\mathcal{H}^2.
\end{equation}
Therefore the assumptions for Theorem \ref{B_2} are satisfied. 
By choosing a suitable subsequence we obtain the following
\begin{equation}
 \label{eq:2_5}
 \begin{array}{c}
 V^0_{f_{m,0}}\rightarrow V^0\mbox{ as oriented varifolds,}\\
 \mu_m\rightarrow \mu\mbox{ weakly as varifolds}\\
 \mu+\mu_0=\mu_{V^0}
 \end{array}
\end{equation}
with 
\begin{equation}
\label{eq:2_5_1}
 V^0=V^0(M,\theta_+,\theta_-,\xi)
\end{equation}
being an integral oriented varifold as in \eqref{eq:B_3}.
Since the topology of $f_m$ is fixed, we also obtain a bound on the second fundamental form (see e.g. \cite[Eq. (1.1)]{Schaetzle}) i.e.
\begin{equation*}
 \int_{\Sigma_\oplus}|A_{m,0}|^2\, d\mu_{g_m\oplus g_0}\leq C:=C(E, \Sigma_0,\lambda,H_0).
\end{equation*}
Therefore \cite[Theorem 5.3.2]{Hutchinson} yields $\mu+\mu_0$ to be an integral $2$-varifold with weak second fundamental form $A_{\mu+\mu_0}\in L^2(\mu+\mu_0)$ 
and as in \cite[Eq. (2.6)]{Schaetzle} we obtain for $A_\mu:=A_{\mu+\mu_0}-A_{\mu_0}\in L^2(\mu)$
\begin{equation}
 \label{eq:2_6}
 \begin{array}{c}
  |A_{m,0}|^2(\mu_m+\mu_0)\rightarrow \nu_0\mbox{ weakly as Radon measures,}\\
  |A_m|^2\mu_m\rightarrow \nu\mbox{ weakly as Radon measures,}\\
  |A_\mu|^2\mu\leq \nu\leq \nu_0,\ \nu_0(\R^n)\leq C:=C(\Gamma,{\bf n}, \Sigma,\lambda,H_0).
 \end{array}
\end{equation}

\begin{lemma}
 \label{2_1}
 $spt(\mu+\mu_0)$ is compact.
\end{lemma}
\begin{proof}
 Simon's diameter estimate \cite[Lemma 1.1]{Simon} and \eqref{eq:2_4} yield
 \begin{equation}
 \label{eq:2_7}
  \operatorname{diam}(f_{m,0}(\Sigma_\oplus))\leq C\sqrt{(\mu_{m}+\mu_0)(\R^n)\cdot W_{0,0}(V^0_{f_{m,0}})}\leq C.
 \end{equation}
Now let $x\in spt(\mu+\mu_0)$. For an arbitrary $\rho>0$ we obtain by e.g. \cite[Prop. 4.26]{Maggi} and the defintion of the support of a Radon measure
\begin{equation*}
 0<(\mu+\mu_0)(B_\rho(x))\leq \liminf_{m\rightarrow\infty}(\mu_m+\mu_0)(B_\rho(x)).
\end{equation*}
Hence $spt(\mu_m+\mu)\cap B_\rho(x)\neq\emptyset$ for $m$ big enough. Therefore we can find $x_m\in spt(\mu_m+\mu_0)$ such that $x_m\rightarrow x$. By \eqref{eq:2_7} we finally obtain 
\begin{equation*}
 \operatorname{diam}(spt(\mu+\mu_0))\leq C
\end{equation*}
and the lemma is proven.
\end{proof}

We also obtain the following lemma because the needed assumptions in \cite[Prop. 2.1]{Schaetzle} are also satisfied:
\begin{lemma}[see Prop. 2.1 in \cite{Schaetzle}]
 \label{2_2}
 \begin{equation*}
  f_{m,0}(\Sigma_\oplus)=spt(\mu_m+\mu_0)\rightarrow spt(\mu+\mu_0)\supset\Gamma\neq \emptyset
 \end{equation*}
 locally in Hausdorff distance, that is
 \begin{equation*}
  spt(\mu + \mu_0)=\{x\in\R^n:\ \exists x_m\in spt(\mu_m+\mu_0)\mbox{ with } x_m\rightarrow x\}.
 \end{equation*}
\end{lemma}

\section{Partial Regularity}
\label{sec:3}
In this chapter we will show $C^{1,\alpha}$ regularity close to good points. 
Our proof strongly follows the argument of Schätzle in \cite[Prop. 2.2]{Schaetzle} respectively Simon in \cite[pp. 298-303]{Simon}. It needs some modifications, 
which will be highlighted in the exposition. We also repeat some details, which we will have to refer to in section \ref{sec:4}, when proving lower semicontinuity.\\
A good point $x_0\in spt(\mu+\mu_0)$ is defined as essentially having
\begin{equation}
\label{eq:3_1}
 \nu_0(x_0)< \varepsilon_0^2
\end{equation}
for an $\varepsilon_0>0$ small enough. \eqref{eq:2_6} yields only finitely many bad points, i.e. points which do not satisfy the following requirement. The precise proposition is as follows.

\begin{lemma}[cf. Prop. 2.2 in \cite{Schaetzle}]
 \label{3_1}
 For any $\varepsilon>0$ there exist $\varepsilon_0 = \varepsilon_0(E,{\bf n},\Gamma,\varepsilon)>0$, $\theta=\theta(E,{\bf n},\Gamma,\varepsilon)>0$, $\rho_0=\rho_0(E,{\bf n},\Gamma,\varepsilon)>0$,
 $\beta=\beta(E,{\bf n},\Gamma)>0$, such that for every good point $x_0\in spt(\mu+\mu_0)$ and good radius $0<\rho_{x_0}\leq \rho_0$ satisfying
 \begin{equation}
 \label{eq:3_2}
  \nu_0(\overline{B_{\rho_{x_0}}(x_0)}) < \varepsilon_0^2,
 \end{equation}
 $\mu+\mu_0$ is a union of $(W^{2,2}\cap C^{1,\beta})$-graphs in $B_{\theta\rho_{x_0}}(x_0)$ 
 of functions $u_i\in (W^{2,2}\cap C^{1,\beta})(B_{\theta \rho_{x_0}}(x_0)\cap L_i)$. 
 Here $L_i\subset \R^3$ are two dimensional affine spaces and $i=1,\ldots, I_{x_0}\leq C(E,\lambda, H_0)$. 
 Furthermore the $u_i$ satisfy the following estimate
 \begin{align}
  \begin{split}
  \label{eq:3_3}
  &(\theta \rho_{x_0})^{-1}\|u_i\|_{L^\infty\left(B_{\theta\rho_{x_0}}(x_0)\cap L_i\right)}\\
  +\ & \|\nabla u_i\|_{L^\infty\left(B_{\theta\rho_{x_0}}(x_0)\cap L_i\right)} + (\theta\rho_{x_0})^\beta h\ddot{o}l_{B_{\theta\rho_{x_0}}(x_0)\cap L_i,\beta}\nabla u_i\leq \varepsilon.
 \end{split}
 \end{align}
Moreover we have a power-decay for the second fundamental form, i.e. $\forall x\in B_{\frac{\theta\rho_{x_0}}{4}}(x_0)$, $0<\rho<\frac{\theta \rho_{x_0}}{4}$
\begin{equation}
 \label{eq:3_4}
 \int_{B_\rho(x)}|A_{\mu+\mu_0}|^2\, d(\mu+\mu_0)\leq C(E,{ \bf n},\Gamma)(\varepsilon_0^2+\rho_0^2)\rho^\beta\rho_{x_0}^{-\beta}.
\end{equation}
\end{lemma}

\begin{proof}
Our goal is to verify the assumptions of Allard's regularity Theorem \ref{A_1}. To do this we need to decompose $\mu+\mu_0$ into parts with Hausdorff density $1$. For this we need
Simon's graphical decomposition Theorem for immersions \ref{A_6} (cf. \cite[Lemma 2.1]{Simon}).
 By \eqref{eq:3_2} and the upper semicontinuity for measure convergence evaluated for closed sets (see e.g. \cite[Prop. 4.26]{Maggi}) we have
 \begin{equation}
  \label{eq:3_5}
  \limsup_{m\rightarrow\infty}\int_{B_{\rho_{x_0}}(x_0)}|A_{m,0}|^2\, d(\mu_m+\mu_0)\leq \nu_0(\overline{B_{\rho_{x_0}}(x_0)}) < \varepsilon_0^2
 \end{equation}
and we can apply the decomposition Theorem \ref{A_6} to $f_{m,0}$ for $m$ big enough. 
Therefore we decompose $f_{m,0}^{-1}(\overline{B_{\frac{\rho_{x_0}}{2}}(x_0)})$
for large $m$ and $\varepsilon_0=\varepsilon_0(E)>0$ small enough into closed pairwise disjoint sets $D_{m,i}\subset \Sigma_\oplus$, $i=1,\ldots,I_m$, $I_m\leq C E$:
\begin{equation*}
 f_{m,0}^{-1}(\overline{B_{\frac{\rho_{x_0}}{2}}(x_0)})=\sum_{i=1}^{I_m} D_{m,i} = \dot \bigcup_{i=1}^{I_m}D_{m,i}.
\end{equation*}
More precisely there exist affine $2$-planes $L_{m,i}\subset \R^3$ and smooth functions $u_{m,i}:\overline{\Omega_{m,i}}\subset L_{m,i}\rightarrow L_{m,i}^\perp$, $i=1,\ldots, I_m$. 
Here $\Omega_{m,i}=\Omega^0_{m,i}\setminus \cup_k d_{m,i,k}$, $\Omega^0_{m,i}$ are simply connected and open and $d_{m,i,k}$ are closed, pairwise disjoint, topological discs such that the $u_{m,i}$ satisfy
\begin{equation*}
 \rho_{x_0}^{-1}|u_{m,i}| + |\nabla u_{m,i}|\leq C(E)\varepsilon_0^{\frac{1}{22}}
\end{equation*}
for all $m,i$. Furthermore we have closed, pairwise disjoint, topological discs $P_{m,i,j}\subset D_{m,i}$, $j=1,\ldots,J_{m,i}$ (see Figure \ref{fig_2}) satisfying
\begin{equation*}
 f_{m,0}\left(D_{m,i}-\bigcup_{j=1}^{J_{m,i}}P_{m,i,j}\right)=\operatorname{graph}(u_{m,i})\cap\overline{B_{\frac{\rho_{x_0}}{2}}(x_0)}
\end{equation*}
and
\begin{equation}
\label{eq:3_6}
 \sum_{i=1}^{I_m}\sum_{j=1}^{J_m} \operatorname{diam} f_{m,0}(P_{m,i,j})\leq C(E)\varepsilon_0^{\frac{1}{2}}\rho_{x_0}.
\end{equation}

As in \cite[Eq. (2.12)-(2.14)]{Schaetzle} we obtain for $\varepsilon_0$ small enough $0<\tau<\frac{1}{2}$ and $0<\theta<\frac{1}{4}$ such that
\begin{equation}
 \label{eq:3_7}
 \frac{\mu_{g_{m,0}}(D_{m,i}\cap f_{m,0}^{-1}(B_{\sigma}(x)))}{w_2\sigma^2}<1+\tau
\end{equation}
for $B_\sigma(x)\subset B_{\theta\rho_{x_0}}(x_0)$ arbitrary. 
Here $w_2$ denotes the Hausdorff measure of the $2$-dimensional euclidean unit ball and $\mu_{g_{m,0}}$ the area measure on $\Sigma_\oplus$ induced by $f_{m,0}$. 
Hence $f_{m,0}|_{D_{m,i}\cap f_{m,0}^{-1}(B_{\theta\rho_{x_0}}(x_0))}$ is an embedding.
As in \cite[Eq. (2.14)-(2.16)]{Schaetzle} the density estimate \eqref{eq:3_7} can be extended to $\mu+\mu_0$. 
We repeat these steps here, because we need the result in section \ref{sec:4}.
Let us define the following Radon measures
\begin{align}
 \begin{split}
 \label{eq:3_8}
 \mu_{m,i}&:=\mathcal{H}^2\lfloor f_{m,0}(D_{m,i}\cap f_{m,0}^{-1}(B_{\theta\rho_{x_0}}(x_0)))\\ &= f_{m,0}\left(\mu_{g_{m,0}}\lfloor (D_{m,i}\cap f_{m,0}^{-1}(B_{\theta\rho_{x_0}}(x_0)))\right).
\end{split}
 \end{align}
\eqref{eq:3_7} shows
\begin{equation}
\label{eq:3_7_1}
 \sum_{i=1}^{I_m}\mu_{m,i} = (\mu_{m}+\mu_0)\lfloor B_{\theta\rho_{x_0}}(x_0).
\end{equation}
Now we take a subsequence depending on $x_0$, $\theta\rho_{x_0}$, may assume $I_m=I$ and get by the usual compactness property of Radon measures as in \cite[Eq. (2.15)]{Schaetzle}
\begin{equation}
 \label{eq:3_9}
 \mu_{m,i}\rightarrow \mu_i\mbox{ weakly as varifolds in }B_{\theta\rho_{x_0}}(x_0).
\end{equation}
Also as in \cite[Eq. (2.15)]{Schaetzle} we obtain
\begin{equation}
\label{eq:3_10}
 spt (\mu_{m,i})\rightarrow spt (\mu_i)\mbox{ locally in Hausdorff distance in }B_{\theta\rho_{x_0}}(x_0)
\end{equation}
for $i=1,\ldots I$. Now pick an arbitrary $\varphi\in C_0^0(B_{\theta\rho_{x_0}}(x_0))$. Then \eqref{eq:3_7_1} yields
\begin{equation}
\label{eq:3_10_1}
 \int\varphi\, d(\mu+\mu_0)\leftarrow\int \varphi\, d(\mu_m+\mu_0) = \sum_{i=1}^I \int \varphi\, d\mu_{m,i}\rightarrow \sum_{i=1}^I\int \varphi\, d\mu_i,
\end{equation}
by approximating the positive and negative part of $\varphi$ monotonically with simple functions, using the monotone convergence theorem and exploiting that the $\mu_{m,i}$ and $\mu_i$ are finite.
The uniqueness part of the Riesz representation theorem now shows
\begin{equation}
 \label{eq:3_11}
 (\mu+\mu_0)\lfloor B_{\theta\rho_{x_0}}(x_0) = \sum_{i=1}^I\mu_i.
\end{equation}
By passing \eqref{eq:3_7} to the limit we also obtain (see \cite[Eq. (2.16)]{Schaetzle})
\begin{equation}
 \label{eq:3_12}
 \frac{\mu_i(B_\sigma(x))}{w_2\sigma^2}\leq 1+\tau\quad \forall\ \overline{B_\sigma(x)}\subset B_{\theta\rho_{x_0}}(x_0). 
\end{equation}
Here we again use the upper semicontinuity for measure convergence evaluated on closed sets (see e.g. \cite[Prop. 4.26]{Maggi}).
Hence assumption \eqref{eq:A_3} of Allard's integral compactness Theorem \ref{A_1} is fullfiled. The next step is to prove a power decay as in \eqref{eq:A_2}.
Then all assumptions on Allard's regularity Theorem \ref{A_1} will be satisfied. 
For this we need to concentrate on a specific $i\in\{1,\ldots,I\}$. 
We also need to make a distinction between boundary points and inner points. 
Let us first assume that $x_0\notin\Gamma$. By the compactness of $\Gamma$ we can additionally assume 
\begin{equation}
 \label{eq:3_13}
0< \rho_{x_0}\leq d(x_0,\Gamma).
\end{equation}
Then $f_{m,0}(D_{m,i})\cap \Gamma =\emptyset$, hence $D_{m,i}\subset \Sigma_0$ or $D_{m,i}\cap \overline{\Sigma_0}=\emptyset$. 
In the first case we do not need to show anything, since $f_{m,0}$ is independent of $m$ on $\Sigma_0$ and smooth on $\Sigma_0$. 
Let us proceed with the inner regularity estimates, i.e. $D_{m,i}\cap \overline{\Sigma_0}=\emptyset$:\\

{\bf Inner regularity:}\\

Let us define 
\begin{equation}
 \label{eq:3_14}
 C_\sigma^{m,i}(x_0):=\{x+y:\ x\in B_\sigma(x_0)\cap L_{m,i}, y\in L_{m,i}^\perp\}.
\end{equation}
Let us choose $0<\rho<\theta\rho_{x_0}$ fixated but arbitrary. 
We need to apply the graphical decomposition Lemma \ref{A_6} again to $f_{m,0}(D_{m,i})\cap \overline{B_\rho(x_0)}$. 
Hence we obtain
smooth functions $v_{m,i,\ell}:\overline{\tilde{\Omega}_{m,i,\ell}}\subset \tilde{L}_{m,i,\ell}\rightarrow \tilde{L}_{m,i,\ell}^\perp$, $\tilde{L}_{m,i}\subset\R^3$ $2$-dimensional planes
($\ell =1,\ldots N_{m,i}\leq CE$),
$\tilde{\Omega}_{m,i}=\tilde{\Omega}^0_{m,i}\setminus\cup_k \tilde{d}_{m,i,\ell,k}$, $\tilde{\Omega}^0_{m,i}$ simply connected and $\tilde{d}_{m,i,\ell,k}$ closed pairwise disjoint discs.
Furthermore we have
\begin{equation}
\label{eq:3_14_1}
 \rho^{-1}|v_{m,i,\ell}| + |\nabla v_{m,i,\ell}|\leq C(E)\varepsilon_0^\frac{1}{22}
\end{equation}
and closed pairwise disjoint discs $\tilde{P}_{m,i,\ell,1},\ldots, \tilde{P}_{m,i,\ell,J_{m,i,\ell}}\subset \tilde{D}_{m,i,\ell}$ such that for all $\ell$
\begin{equation}
\label{eq:3_14_2}
 f_{m,0}\left(\tilde{D}_{m,i,\ell}-\bigcup_{j=1}^{J_{m,i,\ell}}\tilde{P}_{m,i,\ell,j}\right)\cap\overline{B_{\rho}(x_0)}=\operatorname{graph}(v_{m,i,\ell})\cap\overline{B_{\rho}(x_0)}
\end{equation}
These $\tilde{P}_{m,i,\ell,j}$ also satisfy the following estimate
\begin{equation}
 \label{eq:3_15}
 \sum_{j=1}^{J_{m,i,\ell}} \operatorname{diam} f_{m,0}(\tilde{P}_{m,i,\ell,j})\leq C(E)\varepsilon_0^\frac{1}{2}\rho \leq {\frac{1}{8}}\rho,
\end{equation}
if we choose $C(E)\varepsilon_0^\frac{1}{2}<\frac{1}{8}$.
Let us also introduce the corresponding Radon measures similar to \eqref{eq:3_8}
\begin{equation}
\label{eq:3_15_1}
 \tilde{\mu}_{m,i,\ell}:=\mathcal{H}^2\lfloor f_{m,0}(\tilde{D}_{m,i,\ell}) = \mu_{m,i}\lfloor f_{m,0}(\tilde{D}_{m,i,\ell}).
\end{equation}
Since $f_{m,0}|_{D_{m,i}\cap f^{-1}_{m,0}(B_{\theta\rho_{x_0}}(x_0))}$ is an embedding, we also have
\begin{equation}
\label{eq:3_15_2}
 \sum_{\ell=1}^{N_{m,i}} \tilde{\mu}_{m,i,\ell} = \mu_{m,i}\lfloor B_\rho(x_0).
\end{equation}
Inequality \eqref{eq:3_15} yields $\mathcal{L}^1$-measurable sets $S_m\subset (\frac{1}{2}\rho,\frac{3}{4}\rho)$, such that $\forall j=1,\ldots,J_{m,i,\ell}$ (see also Figure \ref{fig_2})
\begin{equation*}
 |S_m|\geq \frac{1}{8}\rho\mbox{ and }\forall \sigma\in S_m:\ \partial C_\sigma^{m,i}(x_0)\cap f(\tilde{P}_{m,i,\ell,j})=\emptyset.
\end{equation*}
\begin{figure}[h] 
\centering 
\includegraphics{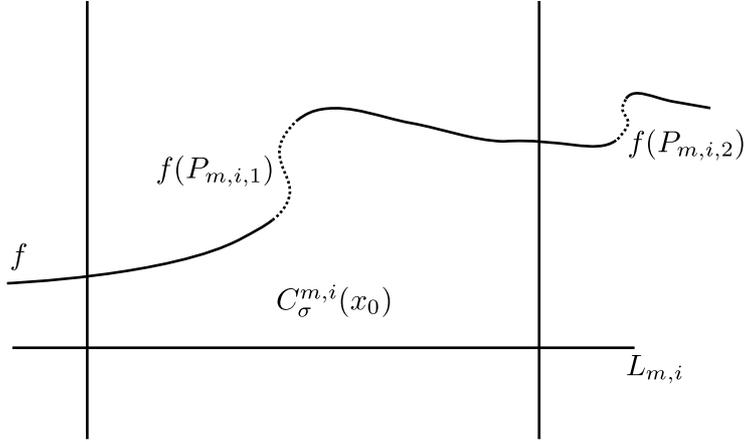}  
\caption{Formation of pimples in graphical decomposition.}
\label{fig_2}
\end{figure}
Therefore $v_{m,i,\ell}|_{\partial B_\sigma(x_0)\cap L_{m,i}}$ and $\nabla v_{m,i,\ell}|_{\partial B_\sigma(x_0)\cap L_{m,i}}$ are well defined for any $\sigma\in S_m$. 
Hence Lemma \ref{A_2} is applicable and yields a biharmonic $w_{m,i,\ell}:B_\sigma(x_0)\cap \tilde{L}_{m,i,\ell}\rightarrow \tilde{L}_{m,i,\ell}^\perp$
with Dirichlet boundary data given by $v_{m,i,\ell}$ and $\nabla v_{m,i,\ell}$. \\
From this point on we will have to deviate from the proof given by Simon \cite[Thm. 3.1]{Simon} resp. Schätzle \cite[Prop. 2.2]{Schaetzle} but our reasoning is still inspired by their argument.\\ 
Let us denote with $A^w_{m,i,\ell}$ the second fundamental form, $H^w_{m,i,\ell}$ the mean curvature vector, $\xi_{w_{m,i,\ell}}$ the orientation and with $K^w_{m,i,\ell}$ the Gauss curvature of $\operatorname{graph}(w_{m,i,\ell})$.
By Gauss-Bonnet $\int_\Sigma K_m\, d\mu_g$ is given entirely by the boundary data (cf. \cite[Remark 2]{DeckGruRoe}). 
Hence $f_{m,0}$ is also a minimising sequence for $f\mapsto W_{H_0,\lambda}(f)+\kappa \int_\Sigma K_f\, d\mu_g$, $\kappa\in\R$ arbitrary. 
Here $K_f$ denotes the Gauss curvature of a given immersion $f:\Sigma\rightarrow\R^3$. Hence for $\sigma\in S_m$ we obtain (see e.g. \cite[Eq. (11)]{DeckGruRoe} for a formula for $|A_m|^2$)
{\allowdisplaybreaks
\begin{align*}
 &\int_{B_{\sigma}(x_0)}|A_m|^2\,d\tilde{\mu}_{m,i,\ell} \\
 =&  \int_{B_{\sigma}(x_0)}|H_m|^2\,d\tilde{\mu}_{m,i,\ell} - 2 \int_{B_{\sigma}(x_0)}K_m\,d\tilde{\mu}_{m,i,\ell}\\
 \overset{\eqref{eq:2_4}}{\leq}& C(H_0,\lambda) \bigg(\int_{B_{\sigma}(x_0)}|H_m-H_0(*\xi_{f_{m,0}})|^2\,d\tilde{\mu}_{m,i,\ell}+ \lambda \tilde{\mu}_{m,i,\ell}(B_\sigma(x_0))\bigg) 
 \\&+  2\int_{B_{\sigma}(x_0)}K_m\,d\tilde{\mu}_{m,i,\ell}\\
 =& C\bigg(\int_{B_{\sigma}(x_0)}|H_m-H_0(*\xi_{f_{m,0}})|^2\,d\tilde{\mu}_{m,i,\ell} + \lambda \tilde{\mu}_{m,i,\ell}(B_\sigma(x_0)) \\
 &+  \frac{2}{C}\int_{B_\sigma(x_0)}K_m\,d\tilde{\mu}_{m,i,\ell}\bigg)\\
 \leq& C\bigg(\int_{\operatorname{graph}(w_{m,i,\ell})}|H^w_{m,i,\ell}-H_0(*\xi_{w_{m,i,\ell}})|^2\, d\mathcal{H}^2\\
 &+ \lambda \mathcal{H}^2(\operatorname{graph}(w_{m,i,\ell})) + \frac{2}{C}\int_{\operatorname{graph}(w_{m,i,\ell})}K^w_{m,i,\ell}\,d\mathcal{H}^2\bigg)+\varepsilon_m\\
 \leq& C\bigg(\int_{\operatorname{graph}(w_{m,i,\ell})}|H^w_{m,i,\ell}|^2\, d\mathcal{H}^2 + \lambda\mathcal{H}^2(\operatorname{graph}(w_{m,i,\ell}))\\
 & + \frac{2}{C}\int_{\operatorname{graph}(w_{m,i,\ell})}K^w_{m,i,\ell}\,d\mathcal{H}^2\bigg)+\varepsilon_m\\
 \overset{\ref{A_5}}{\leq}& C\left(\int_{B_\sigma(x_0)\cap \tilde{L}_{m,i,\ell}}|D^2w_{m,i,\ell}|^2\, dx + \sigma^2 +\sigma\right) + \varepsilon_m\\
 \overset{\ref{A_2}}{\leq}& C\sigma \int_{\graph\left(v_{m,i,\ell}|_{\partial B_\sigma(x_0)\cap \tilde{L}_{m,i,\ell}}\right)}|A_{m,i}|^2d\mathcal{H}^1 + C\sigma + C\sigma^2+\varepsilon_m
\end{align*}}
for $\varepsilon_m\rightarrow 0$. Here we used $|A^w_{m,i,\ell}|^2,|K^w_{m,i,\ell}|\leq C|D^2w_{m,i,\ell}|^2$, which can be seen by the formulas given in e.g. \cite[Subsection 2.1]{DeckGruRoe}.
Integrating over $S_m$ yields with Co-Area formula (see e.g. \cite[Eq. (10.6)]{Simon_Buch})
\begin{equation*}
 \int_{B_{\frac{\rho}{2}}(x_0)}|A_m|^2\, d\tilde{\mu}_{m,i,\ell} \leq C\int_{B_{\frac{3}{4}\rho}(x_0)\setminus B_{\frac{\rho}{2}}(x_0)}|A_m|^2\,d\tilde{\mu}_{m,i,\ell} + C\rho + C\rho^2+\varepsilon_m.
\end{equation*}
Summing over $\ell =1,\ldots N_{m,i}\leq CE$ yields with \eqref{eq:3_15_2}
\begin{equation*}
 \int_{B_{\frac{\rho}{2}}(x_0)}|A_m|^2\, d\mu_{m,i} \leq C\int_{B_{\frac{3}{4}\rho}(x_0)\setminus B_{\frac{\rho}{2}}(x_0)}|A_m|^2\, d\mu_{m,i} + C\rho + C\rho^2+\varepsilon_m.
\end{equation*}
By hole filling, i.e. adding $C$ times the left-handside to the inequality we obtain with  $\gamma:=\frac{C}{C+1}<1$
\begin{equation*}
 \int_{B_{\frac{\rho}{2}}(x_0)}|A_m|^2d\mu_{m,i}\leq\gamma \int_{B_{\frac{3}{4}\rho}(x_0)}|A_m|^2d\mu_{m,i} + C\rho+C\rho^2+\varepsilon_m.
\end{equation*}
For $m\rightarrow\infty$ we can again employ the semicontinuity properties of measure convergence (see e.g. \cite[Prop. 4.26]{Maggi}) and get
\begin{equation*}
 \nu_0(B_{\frac{\rho}{2}}(x_0))\leq \gamma \nu_0(B_\rho(x_0)) + C\rho^2+C\rho.
\end{equation*}
Since all the argument needed was an estimate of the form of 
\begin{equation*}
 \nu(\overline{B_{\rho_{x_0}}(x_0)})\leq \varepsilon_0^2
\end{equation*}
we can repeat the argument for $x\in B_{\frac{\rho_{x_0}}{4}}(x_0)\cap \operatorname{spt}(\mu+\mu_0)$ and $\rho_x:=\frac{\rho_{x_0}}{4}$ (cf. \cite[p. 300]{Simon}). 
Then we obtain for every $0<\rho<\theta \frac{\rho_{x_0}}{4}$
\begin{equation*}
 \nu_0(B_{\frac{\rho}{2}}(x))\leq \gamma \nu_0(B_\rho(x)) + C\rho^2+C\rho
\end{equation*}
with $0<\gamma<1$.
An iteration argument (see e.g. \cite[Lemma 8.23]{GilbargTrudinger}) yields
\begin{equation}
\label{eq:3_17}
 \nu_0(B_\rho(x))\leq C\rho^\beta \rho_{x_0}^{-\beta}(\nu_0(B_{\rho_{x_0}}(x_0)) + \rho_{x_0}^2),\quad \forall 0<\rho<\theta \frac{\rho_{x_0}}{4}
\end{equation}
for some $\beta>0$.
By choosing $\varepsilon_0>0$, $\theta>0$ and $\rho_0>0$ small enough Allard's regularity Theorem \ref{A_1} yields as in \cite[Prop. 2.2, p. 283 bottom]{Schaetzle} $\mu_{i}$ (cf. \eqref{eq:3_9}) to be a $C^{1,\alpha}\cap W^{2,2}$-graph satisfying the estimate \eqref{eq:3_3}.
This concludes the inner regularity.\\

{\bf Regularity at the boundary:}\\

Let us now assume $x_0\in \Gamma$. 
Our reasoning here will be in large parts analogue to the inner regularity, but we will use Theorem \ref{A_3} instead of Theorem \ref{A_2}.
The following preparation for proving an estimate analogue to \eqref{eq:3_17} is identical to \cite[p. 282]{Schaetzle} but we include it nevertheless since we need the notation.
Let $t:\Gamma\rightarrow \partial B_1(0)$ be a smooth tangent of $\Gamma$. 
Since $\Gamma$ consists of smooth embedded pairwise disjoint curves and is compact, there is a $0<\rho_\Gamma(\tau_0)<\infty$ for $\tau_0>0$ independent of $x_0$ satisfying
\begin{equation*}
 |t(x)-t(x_0)|\leq \tau_0\quad \forall x\in\Gamma\cap B_{\rho_\Gamma}(x_0).
\end{equation*}
If we choose $\tau_0$ small enough we even obtain for every $0< \rho\leq \rho_\Gamma$, that
$\Gamma\cap B_\rho(x_0)$ is connected and intersects $\partial B_\rho(x_0)$ transversally.
Hence by assuming $\rho_{x_0}\leq\rho_0\leq\rho_\Gamma$, \eqref{eq:3_7} yields exactly one $i\in\{1,\ldots I\}$ with $\partial \Sigma\cap D_{m,i} \neq\emptyset$. 
Every other $i$ can be dealt with the inner regularity argument, since $D_{m,i}\cap \partial\Sigma=\emptyset$.
By choosing a suitable subsequence w.l.o.g. this $i$ is constant.
By rotating and translating we assume $x_0=0$, $T_{x_0}f_{m,0}(\Sigma)=\R^2\times\{0\}$, $t(x_0)=e_1$ and ${\bf n}(x_0)=e_2$. 
Let us denote with $\pi:\R^3\rightarrow T_{x_0}f_{m,0}(\Sigma)=\R^2\times\{0\}$ the orthogonal projection and let $0<\rho<\theta\rho_{x_0}$ and $0<\sigma<\frac{3}{4}\rho$ be fixated but arbitrary.
Then for $\tau_0$ small enough
\begin{equation}
\label{eq:3_18}
 \pi\left( f_{m,0}(D_{m,i})\cap\Gamma\cap B_\rho(x_0)\cap\pi^{-1}(B_\sigma(x_0))\right)
\end{equation}
is a smooth connected curve in $B_\sigma(x_0)\cap(\R^2\times\{0\})$, which decomposes it and yields $B_\sigma(x_0)\cap(\R^2\times\{0\}) \setminus \pi\left( f_{m,0}(D_{m,i})\cap\Gamma\cap B_\rho(x_0)\cap\pi^{-1}(B_\sigma(x_0))\right)$
to be two connected components $B^\pm_\sigma(x_0)$ ($\pm e_2$ is the inner normal of $B^\pm_\sigma(x_0)$ at $x_0$). 
As in the inner regularity we have to apply the graphical decomposition Lemma \ref{A_6} to $f_{m,0}(D_{m,i})\cap B_\rho(x_0)$ and 
we will also use the same style of notation as in the inner regularity case.
This yields smooth functions $v_{m,i,\ell}$ with the same properties and notations stated in \eqref{eq:3_14_1}-\eqref{eq:3_15_2}.
Since 
\begin{equation*}
 \bigcup_\ell \tilde{D}_{m,i,\ell}\subset D_{m,i}
\end{equation*}
the argument above shows, that exactly for one $\ell$ we have $\tilde{D}_{m,i,\ell}\cap \partial\Sigma\neq\emptyset$.
The other $\ell$ can be dealt with as in the inner regularity case. 
In the rest of the proof we fixate this $\ell$ and only work with this index.
Furthermore we can assume that $\tilde{L}_{m,i,\ell}=\R^2\times\{0\}$, since it is the tangential space of $f_{m,0}$ at $x_0$.
By choosing $\rho_{x_0}$ small enough, $f_{m,0}|_{\tilde{D}_{m,i,\ell}}$ close to $\Gamma$ can be described by a graph satisfying an estimate as in \eqref{eq:3_14_1} with a prefactor of $\frac{1}{2}$. 
Hence we can arrange the pimples $\tilde{P}_{m,i,\ell,j}$ to not coincide with the boundary, i.e.
\begin{equation*}
 f_{m,0}(\tilde{D}_{m,i,\ell})\cap\Gamma\cap B_\rho(x_0)\subset\graph(v_{m,i,\ell})\cap B_\rho(x_0).
\end{equation*}
Otherwise we would use this graph as $v_{m,i,\ell}$.\\
By choosing $C(E)\varepsilon_0^\frac{1}{22}<\frac{1}{8}$, \eqref{eq:3_15} yields a measurable set $S_m\subset]\frac{\rho}{2},\frac{3\rho}{4}[$ with $\mathcal{L}^1(S_m)\geq\frac{\rho}{8}$ 
depending on $x_0$ and $\rho$ such that $\forall \sigma\in S_m$
\begin{equation*}
 f_{m,0}(\tilde{P}_{m,i,\ell,j})\cap \pi^{-1}(\partial B_\sigma^+(x_0))=\emptyset,\quad j=1,\ldots J_{m,i,\ell}.
\end{equation*}
The decomposition by \eqref{eq:3_18} also gives us
\begin{align*}
 f_{m,0}^{-1}(\graph(v_{m,i,\ell})\cap\pi^{-1}(B_\sigma^+(x_0)))\cap \tilde{D}_{m,i,\ell}&\subset \Sigma\\
 f_{m,0}^{-1}(\graph(v_{m,i,\ell})\cap\pi^{-1}(B_\sigma^-(x_0)))\cap \tilde{D}_{m,i,\ell}&\subset \mathring{\Sigma}_0.
\end{align*}
Hence Lemma \ref{A_3} yields a $w_m\in C^2(\overline{B^+_\sigma(x_0)})$ for $\sigma\in S_m$, which satisfies
\begin{equation*}
 w_m=v_{m,i,\ell},\ \nabla w_m=\nabla v_{m,i,\ell}\mbox{ on }\partial B_\sigma^+(x_0)
\end{equation*}
and the estimates \eqref{eq:A_6} and \eqref{eq:A_7}. By \eqref{eq:3_14_1} we get $|\nabla v_{m,i,\ell}|\leq 1$ for $\varepsilon_0>0$ small enough. 
Therefore we obtain $|D^2u_m|\leq C|A_m|$. 
As in the inner regularity case we will compare the Helfrich energy of $v_{m,i,\ell}$ to $w_{m}$ and denote the curvatures of $w_m$ by $A^w_{m}$, $K^w_{m}$, $H^w_{m}$. 
Analogue to the inner regularity case we obtain
{\allowdisplaybreaks
\begin{align*}
 &\int_{\pi^{-1}(B^+_{\sigma}(x_0))}|A_m|^2\,d\tilde{\mu}_{m,i,\ell} \\
 =&  \int_{\pi^{-1}(B^+_{\sigma}(x_0))}|H_m|^2\,d\tilde{\mu}_{m,i,\ell} - 2 \int_{\pi^{-1}(B^+_{\sigma}(x_0))}K_m\,d\tilde{\mu}_{m,i,\ell}\\
 \overset{\eqref{eq:2_4}}{\leq}& C(H_0,\lambda) \bigg(\int_{\pi^{-1}(B^+_{\sigma}(x_0))}|H_m-H_0(*\xi_{f_{m,0}})|^2\,d\tilde{\mu}_{m,i,\ell}\\ 
 &+ \lambda \tilde{\mu}_{m,i,\ell}(\pi^{-1}(B^+_{\sigma}(x_0)))\bigg) +  2\int_{\pi^{-1}(B^+_{\sigma}(x_0))}K_m\,d\tilde{\mu}_{m,i,\ell}\\
 =& C\bigg(\int_{\pi^{-1}(B^+_{\sigma}(x_0))}|H_m-H_0(*\xi_{f_{m,0}})|^2\,d\tilde{\mu}_{m,i,\ell}\\
 &+ \lambda \tilde{\mu}_{m,i,\ell}(\pi^{-1}(B^+_{\sigma}(x_0))) +  \frac{2}{C}\int_{\pi^{-1}(B^+_{\sigma}(x_0))}K_m\,d\tilde{\mu}_{m,i,\ell}\bigg)\\
 \leq& C\bigg(\int_{\operatorname{graph}(w_{m})}|H^w_{m}-H_0(*\xi_{w_{m}})|^2\,d\mathcal{H}^2\\
 &+ \lambda \mathcal{H}^2(\operatorname{graph}(w_{m})) + \frac{2}{C}\int_{\operatorname{graph}(w_{m})}K^w_{m}\,d\mathcal{H}^2\bigg)+\varepsilon_m\\
 \leq& C\bigg(\int_{\operatorname{graph}(w_{m})}|H^w_{m,i}|^2\,d\mathcal{H}^2 + \mathcal{H}^2(\operatorname{graph}(w_{m}))\\
 &+ \frac{2}{C}\int_{\operatorname{graph}(w_{m})}K^w_{m}\,d\mathcal{H}^2\bigg)+\varepsilon_m\\
 \overset{\eqref{eq:A_6}}{\leq}& C\left(\int_{B_\sigma^+(x_0)}|D^2w_{m}|^2\, dx + \sigma^2 \right) + \varepsilon_m\\
 \overset{\eqref{eq:A_7}}{\leq}& C\sigma \int_{f_{m,0}(\tilde{D}_{m,i,\ell})\cap\pi^{-1}(\partial B_\sigma^+(x_0))}|A_{m}|^2\,d\mathcal{H}^1 + C\sigma^2+\varepsilon_m
\end{align*}}
As in the inner regularity case by integrating over $S_m\subset ]\frac{\rho}{2},\frac{3\rho}{4}[$ and using Co-Area formula we obtain
\begin{equation*}
\int_{\pi^{-1}\left(B^+_{\frac{\rho}{2}}(x_0)\right)}|A_m|^2\,d\tilde{\mu}_{m,i,\ell}\leq C \int_{\pi^{-1}\left(B^+_{\frac{3\rho}{4}}(x_0)\setminus B^+_{\frac{\rho}{2}}(x_0)\right)}|A_m|^2\,d\tilde{\mu}_{m,i,\ell} + C\rho^2 + \varepsilon_m.
\end{equation*}
Since the other $\ell$ have been dealt with the interior argument, we obtain
\begin{equation*}
 \int_{\pi^{-1}\left(B^+_{\frac{\rho}{2}}(x_0)\right)}|A_m|^2\,d\mu_{m,i}\leq C \int_{\pi^{-1}\left(B^+_{\frac{3\rho}{4}}(x_0)\setminus B^+_{\frac{\rho}{2}}(x_0)\right)}|A_m|^2\,d\mu_{m,i} + C\rho^2 + \varepsilon_m.
\end{equation*}
The rest of the proof is now the same as \cite[Prop. 2.2, p. 283 bottom half]{Schaetzle}.
\end{proof}

\begin{remark}
 \label{3_2}
 An inspection of the arguments in the proof of Lemma \ref{3_1} yields, that the specific form of the functional $W_{H_0,\lambda}$ does not matter. 
 For Dirichlet boundary value problems as in \eqref{eq:1_3} a functional $F$ for a smooth immersion $f:\Sigma\rightarrow \R^3$ just has to satisfy
 \begin{equation*}
  C_1\int_\Sigma |H_f|^2+1\,d\mu_g \leq F(f)\leq  C_2\int_\Sigma |A_f|^2 +1\, d\mu_g 
 \end{equation*}
for some constants $C_1,C_2>0$. Then the result of Lemma \ref{3_1} would already follow.\\
Here $A_f$ is the second fundamental form, $H_f$ the mean curvature vector and $\mu_g$ the area measure induced by $f$.
\end{remark}

\section{Lower Semicontinuity}
\label{sec:4}
Our main lemma in this section shows lower semicontinuity of the minimising sequence of section \ref{sec:2}. 
Unfortunately we cannot expect this result to be true in general because Gro\ss e-Brauckmann constructed counterexamples in \cite{GrosseBrauck}.
\begin{lemma}
 \label{4_1}
 The minimising sequence $V^0_{f_{m,0}}$, see \eqref{eq:2_5}, satisfies the following lower semi-continuity property
 \begin{equation*}
W_{H_0,\lambda}(V^0)\leq\liminf_{m\rightarrow\infty} W_{H_0,\lambda}(V^0_{f_{m,0}}). 
 \end{equation*}
\end{lemma}
\begin{proof}
Let $x_0\in spt(\mu+\mu_0)$ be a good point (see Lemma \ref{3_1}), i.e. there exists a $\rho_0>0$, such that for every $0<\rho_{x_0}\leq \rho_0$ we have
\begin{equation*}
 \nu_0(\overline{B_{\rho_{x_0}}(x_0)})<\varepsilon_0^2.
\end{equation*}
First we choose an appropiate subsequence, such that $\liminf_{m\rightarrow\infty} W_{H_0,\lambda}(V^0_{f_{m,0}})=\lim_{m\rightarrow\infty} W_{H_0,\lambda}(V^0_{f_{m,0}})$ and relabel if necessary.
As in the proof of Lemma \ref{3_1} inequality \eqref{eq:3_5} allows us to apply the graphical decomposition Lemma \ref{A_6} to $f_{m,0}({B_{\frac{x_0}{2}}(x_0)})$. 
Hence we find closed pairwise disjoint sets $D_{m,i}\subset \Sigma_\oplus$, $i=1,\ldots I_m$, $I_m\leq CE$ such that
\begin{equation*}
 f_{m,0}^{-1}(B_{\frac{\rho_{x_0}}{2}}(x_0))=\sum_{i=1}^{I_m} D_{m,i}.
\end{equation*}
Now \eqref{eq:3_7} yields a $0<\theta< \frac{1}{4}$ such that
\begin{equation}
\label{eq:4_0_0}
 f_{m,0}:{D_{m,i}\cap f_{m,0}^{-1}(B_{\theta\rho_{x_0}}(x_0))}\rightarrow \R^3\mbox{ is an embedding.}
\end{equation}

Since $I_m\leq CE$ we can choose another subsequence dependend on $\rho_{x_0}$, $\theta$ and $x_0$, such that $I_m=I$ is independent of $m$. Of course we relabel again if necessary.
Let us now define analogously to \eqref{eq:3_8} oriented varifolds corresponding to the $D_{m,i}$ in $B_{\theta\rho_{x_0}}(x_0))$. 
Hence we also need an orientation. Let $\tau:\Sigma_\oplus\rightarrow G^0(2,3)$ be the given orientation of $\Sigma_\oplus$. Now let 
\begin{equation}
 \label{eq:4_0}
 \xi_{m,i}:B_{\theta\rho_{x_0}}(x_0)\cap f(D_{m,i})\rightarrow\R^3
\end{equation}
be defined by
\begin{equation}
 \label{eq:4_0_1}
 \xi_{m,i}(x):= df_{m,0}\left(\tau\left(\left(f_{m,0}|_{D_{m,i}\cap f^{-1}_{m,0}(B_{\theta\rho_{x_0}}(x_0))}\right)^{-1}(x)\right)\right).
\end{equation}
By \eqref{eq:4_0_0} this orientation is well defined.
Now
\begin{equation}
\label{eq:4_1}
 V^0_{m,i}:=V^0(f_{m,0}(D_{m,i})\cap B_{\theta\rho_{x_0}}(x_0),1,0,\xi_{m,0})
\end{equation}
defines an oriented varifold on $\Omega:=B_{\theta\rho_{x_0}}(x_0)$.
Similar to \eqref{eq:3_7_1} we obtain by \eqref{eq:4_0_0} and also using the definition of the densities of $V_{f_{m,0}}^0$ in \eqref{eq:2_1_1}
\begin{equation}
\label{eq:4_1_1}
 V^0_{f_{m,0}}\lfloor \left(B_{\theta\rho_{x_0}}(x_0)\times G^0(2,3)\right)=\sum_{i=1}^I V^0_{m,i}.
\end{equation}
Before we proceed let us notate the corresponding masses by
\begin{equation*}
 \mu_{m,i}:=\mu_{V^0_{m,i}},
\end{equation*}
which coincide with the definition given in \eqref{eq:3_8}
\begin{align}
 \begin{split}
 \label{eq:4_1_2}
 \mu_{m,i}&=\mathcal{H}^2\lfloor f_{m,0}(D_{m,i}\cap f_{m,0}^{-1}(B_{\theta\rho_{x_0}}(x_0))),
\end{split}
 \end{align}
 because one of the densitity of $V^0_{m,i}$ is one and the other is zero.
 Since $f_{m,0}:\Sigma_\oplus\rightarrow \R^3$ is a closed immersion, the boundary current satisfies $\partial[|V^0_{m,i}|]=0$ on $\Omega$.
 The first variation and the mass are bounded as well and hence we can apply Hutchinson's compactness result \ref{B_2} to the $V^0_{m,i}$.
 By extracting a suitable subsequence dependend again on $\theta$, $x_0$ and $\rho_{x_0}$, we obtain after relabeling
\begin{align}
\begin{split}
\label{eq:4_2}
 V^0_{m,i}\rightarrow V^0_i\mbox{ as oriented varifolds on } G^0(\Omega),&\\
 \mu_{m,i}\rightarrow \mu_i\mbox{ weakly as varifolds on } \Omega.&
 \end{split}
\end{align}
Now let $\Phi:C_0^0(\Omega\times G^0(2,3))\rightarrow\R$ be arbitrary. As in \eqref{eq:3_10_1} we obtain
\begin{equation*}
 \int\Phi\, \, dV^0\leftarrow \int\Phi\, dV^0_{m}=\sum_{i=1}^I \int\Phi\, dV^0_{m,i} \rightarrow \sum_{i=1}^I\int \Phi\, dV^0_i
\end{equation*}
by approximating $\Phi$ monotonically by simple functions, using Beppo-Levi's theorem and exploiting the finiteness of the $V^0_{m,i}$.
Riesz representation theorem again yields 
\begin{equation*}
 V^0\lfloor (B_{\theta\rho_{x_0}}(x_0)\times G^0(2,3))=\sum_{i=1}^I V^0_{i}.
\end{equation*}
The proof of Lemma \ref{3_1} yield the $\mu_i$ to be $C^{1,\alpha}\cap W^{2,2}$ graphs. 
More precisely there exist affine $2$-planes $L_i$ and $u_i\in C^{1,\alpha}\cap W^{2,2}(L_i\cap B_{\theta\rho_{x_0}}(x_0),L_i^\perp)$, such that
\begin{equation*}
 \mu_i=\mathcal{H}^2\lfloor\{y\in\R^3|\ \exists x\in L_i\cap B_{\theta\rho_{x_0}}(x_0)\mbox{ with } y=u_i(x)+x\}.
\end{equation*}
For simplicities sake we call
\begin{equation}
\label{eq:4_2_0}
 \graph{u_i}:=\{y\in\R^3|\ \exists x\in L_i\cap B_{\theta\rho_{x_0}}(x_0)\mbox{ with } y=u_i(x)+x\}.
\end{equation}
Hence we find densities $\theta_+^i,\theta_-^i:\graph{u_i}\rightarrow\N_0$ with $\theta_+^i+\theta_-^i=1$, such that
\begin{equation*}
 V^0_i=V^0(\graph(u_i),\theta_+^i,\theta_-^i,\xi_i).
\end{equation*}
Here $\xi_i:\graph(u_i)\rightarrow G^0(2,3)$ denotes the $\mathcal{H}^2$-measurable orientation of $V^0_i$.
We need to show, that $\xi_i$ is continuous:
Let $\pi_{L_i}:\R^3\rightarrow L_i$ be given by the following operation: Every $y\in \R^3$ can be decomposed uniquely into $y^\parallel\in L_i$ and $y^\perp\in L_i^\perp$ by $y=y^\parallel+y^\perp$.
Then we set $\pi_{L_i}(y)=y^\parallel$. 
We call this function the orthogonal projection onto $L_i$.
Then $\pi_{L_i}|_{\graph(u_i)}\rightarrow L_i\cap B_{\theta\rho_{x_0}}(x_0)$ is given by $(u_i(x)+x))\mapsto x$. 
$\pi_{L_i}|_{\graph(u_i)}\rightarrow L_i\cap B_{\theta\rho_{x_0}}(x_0)$ is bijectiv and $\pi_{L_i}$ is Lipschitz with Lipschitz constant smaller or equal than $1$.
Let us examine the projection of this current onto $L_i$ by using the push forward, i.e. $\pi_{L_i\#}[|V^0_{i}|]$ (see \cite[26.20]{Simon_Buch}). 
This push forward is well defined, since $\pi_{L_i}|_{\overline{\graph{u_i}}}$ is proper.
Further $\partial \pi_{L_i\#}[|V^0_{i}|]=0$, since $\partial[|V^0_i|]=0$. 
This comes from the fact, that $\partial[|V^0_{m,i}|]=0$ and that oriented varifold convergence is stronger than current convergence.
Now 
\begin{equation*}
 spt(\pi_{L_i\#}[|V^0_{i}|])\subset L_i
\end{equation*}
and we can employ the constancy theorem for currents (see e.g. \cite[26.27]{Simon_Buch} or \cite[p. 357]{Federer}) and get
\begin{equation*}
 \pi_{L_i\#}[|V^0_{i}|] = c\cdot [|L_i\cap B_{\theta\rho_{x_0}}(x_0)|] 
\end{equation*}
for some constant $c\in\R$. Furthermore $[|L_i\cap B_{\theta\rho_{x_0}}(x_0)|] $ is equipped with a constant orientation $\tau_{L_i}$.
Since $\pi_{L_i}|_{\graph(u_i)}$ is bijective we can project back onto $\graph(u_i)$ and obtain
\begin{equation*}
 [|V^0_i|]=((\pi_{L_i}|_{\graph(u_i)})^{-1})_\#(c\cdot [|L_i\cap B_{\theta\rho_{x_0}}(x_0)|]).
\end{equation*}
Since $u_i$ is continuously differentiable, $(\pi_{L_i}|_{\graph(u_i)})^{-1}:L_i\rightarrow\R^3$ is continuously differentiable as well.
By \cite[p. 138]{Simon_Buch} we obtain
\begin{equation*}
\xi_i(x)= d((\pi_{L_i}|_{\graph(u_i)})^{-1})_{\pi_{L_i}(x)\#}\tau_{L_i},
\end{equation*}
which yields $\xi_i$ to be continuous.
This yields without loss of generality 
\begin{equation}
 \label{eq:4_2_1}
\theta_+^i=1 \mbox{ and } \theta_-^i=0.
\end{equation}
Hence for the mean curvature vector $H$ of $V^0$ we have
\begin{equation}
 \label{eq:4_3}
\int (H(x)-(*\xi)H_0)^2\, dV^0_i(x,\xi) = \int_{\graph(u_i)} (H(x)-(*\xi_i(x))^2\, d\mathcal{H}^2(x)
\end{equation}
Examining the term on the right hand side yields with standard $L^2$ density arguments 
\begin{align}
\begin{split}
\label{eq:4_4}
 \sqrt{\int_{\graph(u_i)} (H-H_0(*\xi_i))^2d\mathcal{H}^2}
 =&\sup\bigg\{\int_{\graph(u_i)}(H-H_0(*\xi_i))\varphi\, d\mathcal{H}^2,\\& \varphi\in C^0_0(\R^3,\R^3), \|\varphi\|_{L^2(\mathcal{H}^2\lfloor\graph(u_i))}\leq 1\bigg\}.
 \end{split}
\end{align}
The next step consists of showing
\begin{equation*}
(*\xi_{m,i})\mu_{m,i}\rightarrow (*\xi_i)\mu_i\ \mbox{ as vector valued Radon measures,}
\end{equation*}
with $\mu_i=\mu_{V^0_i}$ the mass of $V^0_i$.
Let $\Phi:G^0(\R^3)\rightarrow \R$ be defined as
\begin{equation*}
 \Phi(x,\xi)=\varphi(x)\cdot(*\xi)
\end{equation*}
with $\varphi\in C^0_0(\R^3,\R^3)$. Since $*$ is continuous we have $\Phi\in C_0^0(G^0(\R^3),\R)$, which yields with \eqref{eq:4_2}
\begin{equation*}
 \int \varphi(x)\cdot(*\xi)\, dV^0_{m,i}(x,\xi)\rightarrow\int \varphi(x)\cdot(*\xi)\,dV^0_i(x,\xi).
\end{equation*}
The densities given in \eqref{eq:4_1} and \eqref{eq:4_2_1} allow us to reformulate this with the help of the masses
\begin{equation}
 \label{eq:4_5}
 \int \varphi(x)(*\xi_{m,i}(x))d\mu_{m,i}(x)\rightarrow\int \varphi(x)(*\xi_{i}(x))d\mu_i(x).
\end{equation}
Let us call $H_{m,i}$ and $H_{i}$ the mean curvature vectors of $V^0_{m,i}$ respectively $V^0_i$.
Since $\mu_{m,i}$ are smooth, we have $H_{m,i}(x)=H_m(x)$ for $\mathcal{H}^2$-a.e. $x\in spt(\mu_{m,i})\cap spt(\mu_m)$.
Furthermore we have also have $H_i(x)=H(x)$ for $\mathcal{H}^2$-a.e. $x\in spt(\mu_{i})\cap spt(\mu)$, but the reasoning is a bit more involved.
First the functions $u_i$ are twice approximately differentiable by \cite[§6.1, Thm. 4]{EvansGariepy} $\mathcal{H}^2$-a.e., because they are in $W^{2,2}$. 
Then \cite[Thm 4.1]{SchaetzleCurrentLower} is applicable and we obtain an explicit formula for the mean curvature by the parameterization of $\graph(u_i)$ given in \eqref{eq:4_2_0}.
This yields the desired result.
Since the first variation is continuous with respect to varifold convergence, we obtain 
\begin{align}
\begin{split}
\label{eq:4_5_1}
 &(H_{m}\pm H_0(*\xi_{m,i}))\mu_{m,i}=(H_{m,i}\pm H_0(*\xi_{m,i}))\mu_{m,i}\\ \rightarrow\quad &(H_i\pm H_0(*\xi_i))\mu_i=(H\pm H_0(*\xi_i))\mu_i
\end{split}
 \end{align}
 as vector valued Radon measures.
Let $\varepsilon>0$ be arbitrary. 
Then \eqref{eq:4_4} yields a continuous $\varphi\in C_0^0(\R^3,\R^3)$, such that $\|\varphi\|_{L^2(\mathcal{H}^2\lfloor\graph(u_i))}=\|\varphi\|_{L^2(\mu_i)} \leq 1$ and
\begin{align*}
 &\sqrt{\int_{\graph(u_i)} (H-H_0(*\xi_i))^2d\mathcal{H}^2}\leq \int_{\graph(u_i)}(H-H_0(*\xi_i))\varphi\, d\mathcal{H}^2 + \varepsilon \\
 \overset{\eqref{eq:4_5_1}}{=}&\liminf_{m\rightarrow\infty}\int(H_m-H_0(*\xi_{m,i}))\varphi\, d\mu_{m,i} + \varepsilon\\
 \leq& \liminf_{m\rightarrow\infty}\sqrt{\int(H_m-H_0(*\xi_{m,i}))^2d\mu_{m,i}}\lim_{m\rightarrow\infty}\sqrt{\int\varphi^2 d\mu_{m,i}} + \varepsilon\\
 =&\liminf_{m\rightarrow\infty} \sqrt{\int(H_m-H_0(*\xi_{m,i}))^2 d\mu_{m,i}}\sqrt{\int\varphi^2 d\mu_{i}}+\varepsilon\\
 \leq&\liminf_{m\rightarrow\infty} \sqrt{\int(H_m-H_0(*\xi_{m,i}))^2 d\mu_{m,i}} + \varepsilon.
\end{align*}
By the Lemmas \ref{2_1} and \ref{2_2} the mass is continuous, i.e. for all $i=1,\ldots I$
\begin{equation*}
 \mu_{m,i}(\R^3)\rightarrow \mu_i(\R^3)\mbox{ and } \mu_{m}(\R^3)\rightarrow \mu(\R^3).
\end{equation*}
Together with \eqref{eq:4_5_1} this yields the lower semi-continuity of the Helfrich energy for one graph:
\begin{equation}
 \label{eq:4_6}
 W_{H_0,\lambda}(V^0_i)\leq \liminf_{m\rightarrow \infty}W_{H_0,\lambda}(V^0_{m,i}).
\end{equation}
Since the mean curvature vector is locally unique, i.e. $H_i=H$ $\mathcal{H}^2$-a.e. on $\graph(u_i)$ we can add these varifolds up and obtain
\begin{equation*}
 \sum_{i=1}^I W_{H_0,\lambda}(V_i^0)=W_{H_0,\lambda}(V^0)\big|_{B_{\theta\rho_{x_0}}(x_0)}.
\end{equation*}
Here the notation $\big|_{B_{\theta\rho_{x_0}}(x_0)}$ means integration over this set, i.e. 
\begin{equation*}
 W_{H_0,\lambda}(V^0)\big|_{B_{\theta\rho_{x_0}}(x_0)}:= \int_{G^0(B_{\theta\rho_{x_0}}(x_0))}(H-H_0(*\xi))^2 dV^0 + \lambda (\mu+\mu_0)(B_{\theta\rho_{x_0}}(x_0)).
\end{equation*}
The argument above is also valid for every $\rho>0$ with $0<\rho\leq\theta\rho_{x_0}$.
By adding the varifolds up as mentioned above we obtain the following:\\
For every good point $x_0$ (see Lemma \ref{3_1}) there exists a good radius $\rho_{x_0}$ such that for every $0<\rho\leq\theta\rho_{x_0}$, we have (see also \eqref{eq:4_5_1} for the additivity of the Helfrich energy)
\begin{align}
\begin{split}
\label{eq:4_7}
  &W_{H_0,\lambda}(V^0)\big|_{B_{\rho}(x_0)}= \sum_{i=1}^I W_{H_0,\lambda}(V^0_i)|_{B_{\rho}(x_0)}\\
  \leq& \sum_{i=1}^I\liminf_{m\rightarrow\infty}  W_{H_0,\lambda}(V^0_{m,i})|_{B_{\rho}(x_0)} \leq \liminf_{m\rightarrow\infty}\sum_{i=1}^I W_{H_0,\lambda}(V^0_{m,i})|_{B_{\rho}(x_0)}\\
  =& \liminf_{m\rightarrow \infty}W_{H_0,\lambda}(V^0_{m})\big|_{B_{\rho}(x_0)}.
\end{split}
\end{align}
Let us call the finitely many 'bad' points given by Lemma \ref{3_1} by ${x}_1,\ldots,{x}_K\in spt(\mu+\mu_0)$.
Now we can apply the symmetric Vitali property (see e.g. \cite[Remark 4.5 (2)]{Simon_Buch}) and cover $\R^3\setminus\{x_1,\ldots,{x}_K\}$ 
with pairwise disjoint balls $(B_{\rho_i}(y_i))_{i\in\N}$ $\omega$-a.e., which satisfy \eqref{eq:4_7}. The Radon measure $\omega$ is given by 
\begin{equation}
 \label{eq:4_8}
 \omega(A):=W_{H_0,\lambda}(V^0)\big|_{A}
\end{equation}
Since $\omega(\{{x}_i\})=0$ we can calculate
\begin{align*}
W_{H_0,\lambda}(V^0)=&\omega(\R^3) = \omega(\R^3\setminus\{{x}_1,\ldots,{x}_K\}) = \omega\left(\bigcup_{i\in\N}B_{\rho_i}(y_i)\right)\\
=& \sum_{i\in\N}\omega(B_{\rho_i}(y_i))\\
\leq& \sum_{i\in\N}\liminf_{m\rightarrow\infty}W_{H_0,\lambda}(V^0_{m})\big|_{G^0(B_{\rho_i}(y_i))}\\
\leq& \liminf_{m\rightarrow\infty} \sum_{i\in\N}W_{H_0,\lambda}(V^0_{m})\big|_{G^0(B_{\rho_i}(y_i))}\\
\leq& \liminf_{m\rightarrow\infty} W_{H_0,\lambda}(V^0_{m}).
\end{align*}
Interchanging the sum and $\liminf$ is Fatou's Lemma applied to the counting measure on $\N$.
\end{proof}

\section{Building the branched Helfrich immersion}
\label{sec:5}
In this section we will show that $V^0$ is Helfrich outside of the finitely many bad points and build the corresponding immersion.
Since the needed arguments are essentially given in the papers \cite{Simon}, \cite{Schaetzle} and \cite{NdiayeSchaetzle}, we only sketch the proofs and cite the necessary ideas.

\begin{lemma}[see \cite{Schaetzle} p. 290 or \cite{NdiayeSchaetzle}]
 \label{5_1}
 Let $x_1,\ldots x_K$ be the finitely many bad points (see Lemma \ref{3_1}) of $V^0$. 
 Then for every $\rho>0$ small enough there exists $m\in\N$ big enough such that we find a $C^{1,\alpha}\cap W^{2,2}$-immersion
 \begin{equation*}
  f: \Sigma_{m,\rho}\rightarrow \operatorname{spt}(\mu+\mu_0)\setminus \bigcup_{k=1}^K B_\rho(x_k),
 \end{equation*}
which is surjective. Here
 \begin{equation*}
  \Sigma_{m,\rho}:=f^{-1}_{m,0}\left(\R^3\setminus \bigcup_{k=1}^K B_\rho(x_k)\right).
 \end{equation*}
\end{lemma}
\begin{proof}
 See \cite[p. 290]{Schaetzle} or for a more detailed argument see \cite{NdiayeSchaetzle}.
\end{proof}

\begin{lemma}
 \label{5_2}
 $V^0$ is Helfrich outside the finitely many bad points (see Lemma \ref{3_1}).
\end{lemma}
\begin{proof}
As in Lemma \ref{5_1} let us denote the finitely many bad points with $x_1,\ldots x_K$. 
The arguments in \cite[pp. 311-317]{Simon} only need the Hausdorff distance convergence of $\operatorname{spt}(\mu_m+\mu_0)$ to $\operatorname{spt}(\mu+\mu_0)$ 
(see Lemma \ref{2_2} and also \eqref{eq:3_10}) and a bound on the Willmore energy on the sequence and the limit (see \eqref{eq:2_4} and Lemma \ref{4_1}).
Hence we can also construct comparison immersions as in \cite[p. 317]{Simon} and obtain the following:
For any $\varepsilon>0$ we obtain a $\rho>0$ sufficiently small, such that
there exists radii $0<\tau_1,\ldots,\tau_K\leq \rho$ and immersions $\tilde{f}_m:\Sigma\rightarrow \R^3$ 
respectively $\tilde{f}_{m,0}:\Sigma\oplus \Sigma_0\rightarrow \R^3$, i.e. $\tilde{f}_{m,0}|\Sigma=\tilde{f}_m$, satisfying the boundary conditions \eqref{eq:1_3} for $m$ big enough.
We also have (see \cite[Eq. (3.49)]{Simon})
\begin{equation}
\label{eq:5_1}
 \tilde{f}_{m,0}|_{V_k}=f_{m,0}|_{V_k}
\end{equation}
for some neighbourhood $V_k$ of $f_{m,0}^{-1}(B_{\tau_k}(x_k))$ and all $k=1,\ldots K$.
Furthermore (see \cite[Eq. (3.50)]{Simon})
\begin{equation}
\label{eq:5_2}
 f\big|_{\Sigma_{m,2\rho}}=\tilde{f}_{m,0}\big|_{\Sigma_{m,2\rho}}
\end{equation}
and for $\rho>0$ small enough we obtain by \cite[Eq. (3.51)]{Simon}
\begin{equation}
\label{eq:5_3}
 \int_{\tilde{f}_{m,0}^{-1}(B_{2\rho}(x_k)\setminus B_{\tau_k}(x_k))}|\tilde{A}_m|^2+1\, d\mu_{\tilde{g}}\leq \varepsilon.
\end{equation}
Here $\tilde{A}_{m}$ denotes the second fundamental form of $\tilde{f}_{m,0}$ and $\mu_{\tilde{g}}$ the area measure induced by $\tilde{f}_{m,0}$.\\
By using $\tilde{f}_{m,0}$ we define an arbitrary compact perturbation of $f$ at a good point:
Let $x_0\in (\operatorname{spt}(\mu)\cap f(\Sigma_{m,2\rho}))=(\operatorname{spt}(\mu)\cap\tilde{f}_{m,0}(\Sigma_{m,2\rho}))$ be a good point (cf. Lemma \ref{3_1}). 
Then we find a $\rho_0>0$ such that we can decompose $f(\Sigma_{m,2\rho})\cap B_{\rho_0}(x_0)$ into finitely many $C^{1,\alpha}\cap W^{2,2}$-graphs.
Let us call one of these graphs $u$. Let us now perturb $u$ compactly in $C^{1,\alpha}\cap W^{2,2}$ in $B_{\rho_0}(x_0)$ and call this new graph $u^p$.
This way we define a new immersion $f_{m,0}^p:\Sigma\oplus\Sigma_0\rightarrow \R^3$, which satisfies the boundary conditions and is as $\tilde{f}_{m,0}$ outside of $u^p$ 
(see also Figure \ref{fig_3} for a sketch of the situation).\\
\begin{figure}[h] 
\centering 
\includegraphics{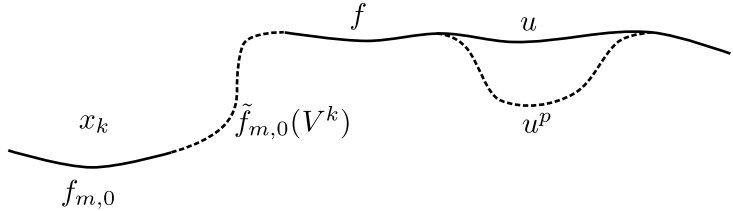}  
\caption{Perturbing $\mu$ outside of the bad points.}
\label{fig_3}
\end{figure}
By the minimising property of $f_{m,0}$ we obtain 
\begin{equation*}
 W_{H_0,\lambda}(f_{m,0})\leq W_{H_0,\lambda}(f^p_{m,0}) + \varepsilon_m,
\end{equation*}
with $\varepsilon_m\rightarrow 0$. Hence by \eqref{eq:5_3}
\begin{equation*}
  W_{H_0,\lambda}(f_{m,0}|_{\Sigma_{m,2\rho}}) \leq W_{H_0,\lambda}(f^p_{m,0}|_{\Sigma_{m,2\rho}}) + C\varepsilon + \varepsilon_m.
\end{equation*}
With the help of our lower semicontinuity result \ref{4_1} we can now let $m\rightarrow \infty$. 
Also note, that $W_{H_0,\lambda}(f^p_{m,0}|_{\Sigma_{m,2\rho}})$ does not depend on $m$, because it is just the Helfrich energy of a perturbed $V^0$ outside of the bad points.
In combination this yields
\begin{equation*}
 W_{H_0,\lambda}(V^0)|_{\R^3\setminus \cup_{k=1}^I B_{2\rho}(x_k)} = W_{H_0,\lambda}(f|_{\Sigma_{m,2\rho}})\leq W_{H_0,\lambda}(f^p_{m,0}|_{\Sigma_{m,2\rho}}) +  C\varepsilon.
\end{equation*}
By \eqref{eq:5_2} we obtain
\begin{equation*}
 W_{H_0,\lambda}(u) \leq W_{H_0,\lambda}(u^p) +C\varepsilon. 
\end{equation*}
Since $\varepsilon>0$ is arbitrary and the inequality does not depend on $\rho$ anymore we finally get
\begin{equation}
 \label{eq:5_4}
 W_{H_0,\lambda}(u) \leq W_{H_0,\lambda}(u^p).
\end{equation}
$u^p$ being a compact perturbation yields $u$ to be Helfrich. 
Since $\varepsilon\rightarrow 0$ implies $\rho\rightarrow0$, $V^0$ is Helfrich outside of the finitely many bad points.
\end{proof}

\begin{lemma}
 \label{5_3}
 The graphs given in Lemma \ref{3_1} are smooth.
\end{lemma}
\begin{proof}
The only assumptions needed for the proof of \cite[Prop. 3.1]{Schaetzle} are for the graph to satisfy
 the estimates \cite[Eq. (3.1), (3.2)]{Schaetzle} 
 and an equation of the form \cite[Eq. (3.4)]{Schaetzle} with the growth conditions \cite[Eq. (3.5)]{Schaetzle}.
 Hence we need to check these conditions and the assumptions of \cite[Lemma 3.2]{Simon}.
 
 Lemma \ref{5_2} shows that our graphs satisfy the Helfrich equation \eqref{eq:1_2} weakly.
 Since the first terms of the Helfrich equation are just the Willmore equation (see e.g. \cite[Eq. (4)]{DeckGruRoe})
 the growth conditions for the highest order already follow by the reasoning of \cite[p. 310]{Simon}.
 The other terms are a quadratic polynomial in the second derivative of the graph (see  e.g. \cite[Eq. (8),(9)]{DeckGruRoe})
 and we have the estimates in Lemma \ref{3_1}. 
 Hence the result follows from \cite[Lemma 3.2]{Simon} and \cite[Prop. 3.1]{Schaetzle}.  
\end{proof}

The next lemma finally proves Theorem \ref{1_1}.
\begin{lemma}[cf. Prop. 4.1 in \cite{Schaetzle}]
 \label{5_4}
 There exists an oriented manifold $\tilde{\Sigma}$ and a branched immersion $f:\tilde{\Sigma}\rightarrow\R^3$ such that $f$
 is Helfrich and satisfies the boundary conditions \eqref{eq:1_3} outside the finitely many branch points. Furthermore $f$ is continuous at the branch points.
\end{lemma}
\begin{proof}
The proof is the same as \cite[Prop. 4.1]{Schaetzle}. 
The first part of it up to \cite[Eq. (4.2)]{Schaetzle} is explained in greater detail in \cite{NdiayeSchaetzle}.
\end{proof}

\appendix
\section{Auxiliary Results}
\label{sec:A}
For the readers convenience we collect a few needed results:

The following  is a variant of Allard's regularity Theorem. 
A proof of this statement can be found in \cite[Section 3]{Simon} or \cite[Korollar 20.3]{Schaetzle_Skript_Allard} (see also \cite[Theorem B.1]{Schaetzle}).
\begin{theorem}[Allard's regularity Theorem, see \cite{Allard}, Theorem 8.16]
\label{A_1}
For $n,m\in\N$, $0<\beta<1$, $\alpha>0$ there exist $\varepsilon_0=\varepsilon_0(n,m,\alpha,\beta)>0$, $\gamma=\gamma(n,m,\alpha,\beta)$ and $C=C(n,m,\alpha,\beta)$ such that:\\
Let $\mu$ be an integral $n$-varifold in $B_{\rho_0}^{n+m}(0)$, $0<\rho_0<\infty$, $0<\varepsilon<\varepsilon_0$ with locally bounded first variation in $B_{\rho_0}^{n+m}(0)$ satisfying
\begin{equation}
 \label{eq:A_1}
 \rho^{1-n}\|\delta \mu\|(B_\rho)\leq \varepsilon^2(\rho^{-n}\mu(B_{\rho}))^{1-\alpha}\rho^{2\beta}\rho_0^{-2\beta},\quad \forall B_{\rho}\subset B_{\rho_0}(0)
\end{equation}
or weak mean curvature $H_{\mu}\in L^2(\mu\lfloor B_{\rho_0}^{n+m}(0))$ satisfying
\begin{equation}
 \label{eq:A_2}
 (\rho^{2-n})\left(\int_{B_\rho}|H_\mu|^2\, d\mu\right)^{\frac{1}{2}}\leq \varepsilon(\rho^{-n}\mu(B_\rho))^{\frac{1}{2}-\alpha}\rho^\beta\rho_0^{-\beta},\quad \forall B_\rho\subset B_{\rho_0}(0)
\end{equation}
and 
\begin{equation}
 \label{eq:A_3}
 0\in\operatorname{spt}\mu,\ \rho_0^{-n}\mu(B_{\rho_0}(0))\leq(1+\varepsilon)\omega_n.
\end{equation}
Then there exists $u\in C^{1,\beta}(B_{\gamma\varepsilon\rho_0}^n(0),\R^m)$ $u(0)=0$, such that after rotation
\begin{equation}
 \label{eq:A_4}
 \mu\lfloor B_{\gamma\varepsilon\rho_0}^{n+m}(0)=\mathcal{H}^n\lfloor(\operatorname{graph} u\cap B_{\gamma\varepsilon\rho_0}^{n+m}(0))
\end{equation}
and
\begin{equation}
 \label{eq:A_5}
 (\varepsilon\rho_0)^{-1}\|u\|_{L^{\infty}(B^n_{\gamma\varepsilon\rho_0}(0))}+\|\nabla u\|_{L^{\infty}(B^n_{\gamma\varepsilon\rho_0}(0))} + (\varepsilon\rho_0)^\beta \operatorname{h\ddot{o}l}_{B_{\gamma\varepsilon\rho_0}^n(0),\beta}\nabla u \leq C\varepsilon^{\frac{1}{2(n+1)}}.
\end{equation}
\end{theorem}

\begin{theorem}[Biharmonic comparison lemma, see \cite{Simon}, Lemma 2.2]
 \label{A_2}
 Let $\Sigma\subset\R^n$ be a smooth embedded $2$-dimensional manifold, $\xi\in\R^n$, $L$ a $2$-dimensional plane containing $\xi$, $u\in C^\infty(U)$ for some open ($L$-)neighbourhood $U$ of $L\cap\partial B_\rho (\xi)$ and
\begin{equation*}
 \operatorname{graph} u\subset\Sigma,\quad |D u|\leq 1.
\end{equation*}
Also let $w\in C^\infty(L\cap \overline{B_\rho(\xi)})$ satisfy
\begin{equation*}
 \left\{\begin{array}{cc}\Delta^2 w=0,& \mbox{ on }L\cap B_\rho(\xi)\\ w=u,\ Dw=Du, &\mbox{ on } L\cap\partial B_\rho(\xi).\end{array}\right.
\end{equation*}
Then 
\begin{equation*}
 \int_{L\cap B_\rho(\xi)}|D^2 w|\, d\mathcal{L}^2\leq C\rho \int_\gamma |A|^2\, d\mathcal{H}^1,
\end{equation*}
where $\gamma=\operatorname{graph}(u|_{L\cap\partial B_\rho(\xi)})$ and $A$ is the second fundamental form of $\Sigma$. $C$ is a fixed constant independent of $\Sigma$ and $\rho$.

\end{theorem}

\begin{theorem}[Trace extension lemma, see \cite{Schaetzle}, Lemma A.1]
 \label{A_3}
Let
\begin{equation*}
 B_{\rho}^+(0):=\{(y,t)\in B_\rho(0)\subset\R^{n-1}\times\R|\ t>\Psi(y)\},
\end{equation*}
where $\Psi\in C^2(\overline{B_\rho^{n-1}(0)})$, $\Psi(0)=0$, $\Psi'(0)=0$, $|\Psi'|\leq \varepsilon_0$ for some $\varepsilon_0$ small enough, $|D^2\Psi|\leq\Lambda$ for some $\lambda<\infty$ and let $u\in C^2(\partial B^+_{\rho}(0))$.\\
Then there exists $w\in C^2(\overline{B_{\rho}^+(0)})$ such that
\begin{equation*}
 w=u,\ \nabla w =\nabla u \mbox{ on }\partial B_\rho^+(0),
\end{equation*}
\begin{equation}
 \label{eq:A_6}
 \rho^{-1}|w|+|\nabla w|\leq C(n,\Lambda)\left(\rho^{-1}\|u\|_{L^\infty(\partial B_\rho^+(0))} + \|\nabla u\|_{L^\infty(\partial B_\rho^+(0))}\right),
\end{equation}
\begin{equation}
\label{eq:A_7}
 \int_{B_{\rho}^+(0)}|D^2 w|^2\, d\mathcal{L}^n\leq C(n,\Lambda)\rho \int_{\partial B_\rho^+(0)}|D^2u|^2\, d\mathcal{H}^{n-1}.
\end{equation}
\end{theorem}

\begin{theorem}[see \cite{Simon}, Lemma 1.1]
 \label{A_4}
 Let $\Sigma\subset\R^n$ be a smooth compact connected $2$-dimensional surface without boundary. Then
 \begin{equation*}
  2\sqrt{\frac{\mathcal{H}^2(\Sigma)}{W_{0,0}(\Sigma)}}\leq \operatorname{diam}\Sigma\leq C\sqrt{\mathcal{H}^2(\Sigma)W_{0,0}(\Sigma)}.
 \end{equation*}
The constant $C=C(n)<\infty$ does not depend on $\Sigma$.
\end{theorem}

The following lemma will be helpful in our first regularity result Theorem \ref{3_1} when combined with Lemma \ref{A_2}. The proof is mainly based on Agmon's estimate \cite[Thm. 1]{Agmon}.
\begin{lemma}
 \label{A_5}
 Let $0<\rho<R$, $B_\rho(0)\subset \R^n$, $u\in C^2(B_\rho(0))$ satisfy $(-\Delta)^2u=0$ in $B_\rho(0)$ and
 \begin{equation*}
  \|u\|_{C^0(\partial B_\rho(0))} + \|\nabla u\|_{C^0(\partial B_\rho(0))}\leq C_0
 \end{equation*}
for a constant $C_0>0$. Then there exists a constant $C_1=C_1(C_0,R,n)>0$ such that
\begin{equation*}
 \|\nabla u\|_{C^0(B_\rho(0))}\leq \rho^{-1}C_1.
\end{equation*}

\end{lemma}
\begin{proof}
 Put $w(x)=u(\rho x)$, then $w$ satisfies $(-\Delta)^2w=0$ in $B_1(0)$. For $x\in B_1(0)$ we also have
 \begin{equation*}
  \nabla w(x)=\rho\nabla u(\rho x).
 \end{equation*}
Agmon's Theorem \cite[Thm. 1]{Agmon} respectively \cite[Eq. (8)]{Agmon} yields a constant $C=C(B_1(0))=C(n)>0$ such that 
\begin{align*}
\|\nabla u\|_{C^0(B_\rho(0))}&= \rho^{-1}\|\nabla w\|_{C^0(B_1(0))}\\
&\leq C \rho^{-1} \|w\|_{C^1(\partial B_1(0))}\\
&\leq C\rho^{-1}(\|u\|_{C^0(\partial B_\rho(0))}+ \rho\|\nabla u\|_{C^0(\partial B_\rho(0))}).
\end{align*}
Hence
\begin{equation*}
 \|\nabla u\|_{C^0(B_\rho(0))}\leq C C_0\max\{1,R\} \rho^{-1},
\end{equation*}
which finishes the proof.
\end{proof}

\begin{lemma}[Graphical decomposition for immersions, cf. \cite{Simon}, Lemma 2.1]
 \label{A_6}
 Let $f:\Sigma\rightarrow \R^n$ be a smooth, $2$-dimensional compact immersion with or without boundary. 
 For any $\beta>0$ there exists an $\varepsilon_0=\varepsilon_0(n,\beta)>0$ (independent of $\Sigma$ and $f$) such that if $\varepsilon\in(0,\varepsilon_0]$,
 $f(\partial \Sigma)\cap B_\rho(x_0)=\emptyset$ for some $x_0\in f(\Sigma)$ and $\rho>0$, also satisfying  $\mu_g(f^{-1}(\overline{B_\rho(x_0)}))\leq \beta \rho^2$ and
 \begin{equation*}
  \int_{f^{-1}(\overline{B_\rho(x_0)})} |A_f|^2\, d\mu_g\leq \varepsilon^2,
 \end{equation*}
 then the following holds:\\
 There exist pairwise disjoint sets $D_{i}\subset \Sigma$ ($i=1,\ldots I$, $I\leq C=C(n, W_{0,0}(f))$), such that
 \begin{equation*}
  f^{-1}(B_{\frac{\rho}{2}}(x_0))=\sum_{i=1}^I D_i.
 \end{equation*}
 Also there are affine $2$-planes $L_i\subset \R^n$ and smooth function $u_i:\overline{\Omega_i}\subset L_i\rightarrow L_i^\perp$ representing $f$. More precisely 
  $\Omega_i=\Omega^0_i\setminus \cup_k d_{i,k}$, $\Omega^0_i$ are simply connected and open and the $d_{i,k}$ are closed pairwise disjoint topological discs. The graphs satisfy
 \begin{equation*}
  \rho^{-1}|u_i| + |\nabla u_i|\leq C(W_{0,0}(f),n)\varepsilon^\frac{1}{4n+10}.
 \end{equation*}
Then there are the so called pimples $P_{i,j}\subset D_i$ ($j=1,\ldots J_i$), which are closed pairwise topological discs and satisfy
\begin{equation*}
 f(D_i\setminus \cup_{j=1}^{J_i}P_{i,j})=\graph(u_i)\cap \overline{B_\rho(x_0)}
\end{equation*}
and
\begin{equation*}
 \sum_{i=1}^I\sum_{j=1}^{J_i}\operatorname{diam} f(P_{i,j})\leq C(W_{0,0}(f),n)\varepsilon^{\frac{1}{2}}\rho.
\end{equation*}

\end{lemma}
\begin{proof}
The proof is explained in \cite[p. 280, top]{Schaetzle} but we sketch it here for the reader's convenience:\\
 By the Whitney embedding theorem we find a smooth embedding $\tilde{f}:\Sigma\rightarrow\R^4$. Let $\tau>0$. 
 Then $(f,\tau \tilde{f}):\Sigma\rightarrow \R^{n+4}$ is an embedding and we can apply Simon's graphical decomposition lemma \cite[Lemma 2.1]{Simon}.
 For $\tau>0$ small we can project $(f,\tau \tilde{f})$ to $\R^n$ and obtain the desired result.
\end{proof}

\section{Oriented varifolds}
\label{sec:B}
Here we collect the basic definitions for oriented varifolds and a compactness theorem, which were both given by Hutchinson (see \cite[Chapter 3]{Hutchinson}).
Let us denote the set of oriented $n$-dimensional subspaces of $\R^{n+m}$ by
\begin{equation}
 \label{eq:B_1}
 G^0(n,n+m)=\{\tau_1\wedge\ldots\wedge\tau_n\subset\Lambda_n\R^{n+m}:\ |\tau_1|=\ldots=|\tau_n|=1,\ \tau_i\perp\tau_j, i\neq j\}.
\end{equation}
Therefore $G^0(n,n+m)$ is a compact metric space. 
Furthermore the Grassmannian manifold of all unoriented $n$-dimensional subspaces of $\R^{n+m}$ is denoted by $G(n,n+m)$ (cf. \cite[Chapter 8]{Simon_Buch}). 
For computational benefits we identify $G(n,n+m)$ with the set of matrices of orthogonal projections onto $n$-dimensional subspaces, i.e.
\begin{equation*}
 G(n,n+m)=\left\{\begin{array}{c} P\in \R^{n+m\times n+m}:\ \operatorname{dim}(P(\R^{n+m}))=n,\\ 
\ P^2=P,\ \forall x,y\in\R^{n+m}\ \langle x, Py\rangle=\langle Px, Py\rangle=\langle Px,y\rangle
\end{array}\right\}.
\end{equation*}
The standard $2$-fold covering map $q_g:G^0(n,n+m)\rightarrow G(n,n+m)$ is given by
\begin{equation*}
 \tau_1\wedge\ldots\wedge\tau_n\mapsto \tau_1\tau_1^T+\ldots+\tau_n\tau_n^T.
\end{equation*}
Here $\tau \tau^T$ denotes the matrix multiplication between $\tau$ and $\tau^T$. Here $\tau$ is a column vector and $\tau^T$ is the transposed.
Please note, that $q_g$ is well defined since the choice of the orthonormal basis $\tau_1,\ldots,\tau_n$ does not matter for the resulting projection.
Let us denote for $\Omega\subset\R^{n+m}$ open
\begin{equation}
 \label{eq:B_2}
 G^0(\Omega):=\Omega\times G^0(n,n+m),\quad G(\Omega):=\Omega\times G(n,n+m).
\end{equation}

\begin{definition}[see \cite{Hutchinson}, page 48]
 \label{B_1}
 An oriented $n$-varifold $V^0$ on an open set $\Omega\subset\R^{n+m}$ is a Radon measure on $G^0(\Omega)$. 
 \end{definition}
 Oriented varifold convergence is defined as follows: $V_k^0\rightarrow V^0$, if and only if for every $\Phi\in C_0^0(G^0(\Omega))$ we have
 \begin{equation*}
  \int_{G^0(\Omega)}\Phi\, dV^0_k\rightarrow \int_{G^0(\Omega)} \Phi\, dV^0.
 \end{equation*}
The projection $\pi^0:G^0(\Omega)\rightarrow\Omega$ given by $(x,\xi)\mapsto x$ defines the mass $\mu_{V^0}$ of an oriented varifold $V^0$ by
\begin{equation*}
 \mu_{V^0}:=\pi^0(V^0)\mbox{, i.e. } \mu_{V^0}(B)= V^0((\pi^0)^{-1}(B)),\ B\subset \Omega.
\end{equation*}
Since $\pi^0$ is proper, $\mu_{V^0}$ is a Radon measure on $\Omega$ (cf. \cite[Appendix A]{Schaetzle_Skript_Allard}). 

Given an oriented $n$-varifold $V^0$ on $\Omega$ the map $q_g$ defines an $n$-varifold $V$ on $\Omega$ by
\begin{equation*}
 V:=(id\times q_g)(V^0),
\end{equation*}
since $(id\times q_g)$ is proper. Furthermore $V^0$ defines an $n$-dimensional current on $\Omega$ by (cf. \cite[Chapter 6]{Simon_Buch} for more informations on currents)
\begin{equation*}
 [|V^0|](\omega):=\int_{G^0(\Omega)}\langle \omega(x), \xi \rangle\, dV^0(x,\xi),\ \omega\in C_0^\infty(\Omega,\Lambda^n\R^{n+m}).
\end{equation*}
 The corresponding mass is denoted by $M_{\Omega}([|V^0|])$ and the boundary current by $\partial[|V^0|]$.
Next we define oriented rectifiable $n$-varifolds. Given a countable $n$-rectifiable subset $M\subset \Omega$, locally $\mathcal{H}^n$ integrable functions
$\theta_+,\theta_-:M\rightarrow [0,\infty]$ and a $\mathcal{H}^n$ measurable function $\xi:M\rightarrow G^0(n,n+m)$, such that $T_xM=q_g(\xi(x))$ $\mathcal{H}^n$-a.e., an oriented rectifiable $n$-varifold is defined by
\begin{equation}
 \label{eq:B_3}
 V^0(\Phi):=V^0(M,\theta_\pm,\xi)(\Phi):=\int_M \Phi(x,\xi(x))\theta_+(x) + \Phi(x,-\xi(x))\theta_-(x)\, d\mathcal{H}^n(x), 
\end{equation}
for $\Phi\in C_0^0(G^0(\Omega))$. 
The set of all oriented rectifiable $n$-varifolds on $\Omega$ is denoted by $RV^0(\Omega)$. 
If $\theta_\pm$ are integer valued, then we say $V^0$ is an oriented integral $n$-varifold on $\Omega$. The set of these varifolds is denoted
by $IV^0(\Omega)$.\\
The first variation of $V^0\in RV^0(\Omega)$ is defined to be the first variation of the unoriented varifold i.e. $\delta V^0=\delta (id\times q_g)(V^0)$ (see e.g. \cite[§ 39]{Simon_Buch}).
We also denote $\delta\mu_{V^0}:=\delta V^0$, since for rectifiable varifolds, $\mu_{V^0}$ already contains the necessary informations for defining curvatures (see e.g. \cite[§ 16]{Simon_Buch}).
We can define the generalized mean curvature vector $H_{V^0}$ of $V^0$ to be the generalized mean curvature vector of $(id\times q_g)(V^0)$, if it is of bounded first variation.
For curvature varifolds (see \cite[Def. 5.2.1]{Hutchinson} for a precise definition) the generalized mean curvature coincides with the trace of the generalized second fundamental form (see \cite[Remark 5.2.3]{Hutchinson}).
Furthermore we write $H_{V^0}\in L^2(\mu_{V^0})$ if for every $X\in C^1_0(\Omega,\R^n)$ we have
\begin{equation}
\label{eq:B_4}
 \delta V^0(X)= -\int_\Omega H_{V^0}\cdot X\, d\mu_{V^0}\ \mbox{ and } \int_\Omega |H_{V^0}|^2\, d\mu_{V^0} < \infty.
\end{equation}
In the sense of \cite[§ 39]{Simon_Buch} this means that $(id\times q_g)(V^0)$ does not have a generalized boundary.\\

Now we can state Hutchinson's compactness result for oriented varifolds:
\begin{theorem}[see \cite{Hutchinson}, Theorem 3.1]
 \label{B_2}
Let $\Omega\subset\R^{n+m}$ be open. The following set is sequentially compact with respect to oriented varifold convergence:
\begin{align*}
 \{V^0\in IV^0(\Omega):\ \forall \Omega'\subset\subset\Omega\ \exists C(\Omega')<\infty:&\\ \mu_{V^0}(\Omega')+\|\delta\mu_{V^0}\|(\Omega') + M_{\Omega'}(\partial[|V^0|])\leq C(\Omega')\}
\end{align*} 
\end{theorem}

\bibliography{bibliography}

\begin{thebibliography}{10}

\bibitem{Agmon}
S.~{Agmon}.
\newblock {Maximum Theorems for Solutions of Higher Order Elliptic Equations}.
\newblock {\em Bull. Amer. Math. Soc.}, 66(2):77--80, 1960.

\bibitem{Allard}
W.~K. {Allard}.
\newblock {On the first variation of a varifold}.
\newblock {\em Ann. Math.}, 95:417--491, 1972.

\bibitem{AnzellottiSerapioniTamanini}
G.~{Anzellotti}, R.~{Serapioni}, and I.~{Tamanini}.
\newblock {Curvatures, Functionals, Currents}.
\newblock {\em Indiana Univ. Math. J.}, 39:617--669, 1990.

\bibitem{BauerKuwert}
M.~{Bauer} and E.~{Kuwert}.
\newblock {Existence of minimizing Willmore surfaces of prescribed genus.}
\newblock {\em Int. Math. Res. Notices}, 10:553--576, 2003.

\bibitem{Brakke}
K.~A. {Brakke}.
\newblock {\em {The motion of a surface by its mean curvature}}.
\newblock Mathematical Notes, Princeton University Press, 1978.

\bibitem{Canham}
P.B. {Canham}.
\newblock The minimum energy of bending as a possible explanation of the
  biconcave shape of the human red blood cell.
\newblock {\em J. Theor. Biol.}, 26(1):61--76, 1970.

\bibitem{ChoksiMorandottiVeneroni}
R.~{Choksi}, M.~{Morandotti}, and M.~{Veneroni}.
\newblock Global minimizers for axisymmetric multiphase membranes.
\newblock {\em ESAIM: COCV}, 19(4):1014--1029, 2013.

\bibitem{ChoksiVeroni}
R.~{Choksi} and M.~{Veneroni}.
\newblock Global minimizers for the doubly-constrained helfrich energy: the
  axisymmetric case.
\newblock {\em Calc. Var. Partial Differ. Equ.}, 48(3):337--366, Nov 2013.

\bibitem{DaLioPalmurellaRiviere}
F.~{Da Lio}, F.~{Palmurella}, and T.~{Rivi\`{e}re}.
\newblock {A Resolution of the Poisson Problem for Elastic Plates}.
\newblock {\em arXiv:1807.09373v1 [math.DG]}, 2018.
\newblock Preprint.

\bibitem{DallDeckGru}
A.~{Dall'Acqua}, K.~{Deckelnick}, and H.-Ch. {Grunau}.
\newblock {Classical solutions to the Dirichlet problem for Willmore surfaces
  of revolution}.
\newblock {\em Adv. Calc. Var.}, 1:379--397, 2008.

\bibitem{DallFroehGruSchie}
A.~{Dall'Acqua}, S.~{Fröhlich}, H.-Ch. {Grunau}, and F.~{Schieweck}.
\newblock {Symmetric Willmore surfaces of revolution satisfying arbitrary
  Dirichlet boundary data}.
\newblock {\em Adv. Calc. Var.}, 4:1--81, 2011.

\bibitem{DeckGruRoe}
K.~{Deckelnick}, H.-Ch. {Grunau}, and M.~{R\"oger}.
\newblock {Minimising a relaxed Willmore functional for graphs subject to
  boundary conditions}.
\newblock {\em Interfaces Free Bound.}, 19:109--140, 2017.

\bibitem{Delladio}
S.~{Delladio}.
\newblock {Special generalized Gauss graphs and their application to
  minimization of functionals involving curvatures}.
\newblock {\em J. reine angew. Math}, 486:17--43, 1997.

\bibitem{Doemeland}
M.~{Doemeland}.
\newblock {Verallgemeinerung eines Existenzsatzes für axialsymmetrische
  Minimierer des Willmore-Funktionals auf das Helfrich-Funktional}.
\newblock Scientific project, Otto-von-Guericke University Magdeburg, 2015.
\newblock Supervisors: H.-Chr. Grunau, J. Wiersig.

\bibitem{EichmannGrunau}
S.~{Eichmann} and H.-Chr. {Grunau}.
\newblock {Existence for Willmore surfaces of revolution satisfying
  non-symmetric Dirichlet boundary conditions}.
\newblock {\em Adv. Calc. Var.}, 2017.
\newblock doi:10.1515/acv-2016-0038.

\bibitem{EvansGariepy}
L.~C. {Evans} and R.~F. {Gariepy}.
\newblock {\em {Measure Theory and Fine Properties of Functions}}.
\newblock CRC Press, 1992.

\bibitem{Federer}
Herbert {Federer}.
\newblock {\em {Geometric Measure Theory}}.
\newblock Springer-Verlag Berlin, 1969.

\bibitem{GilbargTrudinger}
D.~{Gilbarg} and N.~S. {Trudinger}.
\newblock {\em {Elliptic Partial Differential Equations of Second Order}}.
\newblock Springer-Verlag Berlin, 3rd edition, 1998.

\bibitem{GrosseBrauck}
K.~{Gro\ss e-Brauckmann}.
\newblock New surfaces of constant mean curvature.
\newblock {\em Math. Z.}, 214:527--565, 1993.

\bibitem{Helfrich}
W.~{Helfrich}.
\newblock {Elastic properties of lipid bilayers: Theory and possible
  experiments}.
\newblock {\em Z. Naturforsch. C}, 28:693--703, 1973.

\bibitem{Hutchinson}
J.~E. {Hutchinson}.
\newblock {Second Fundamental Form for Varifolds and the Existence of Surfaces
  Minimising Curvature}.
\newblock {\em Indiana U. Math. J.}, 35:45--71, 1986.

\bibitem{LeeIntroManifold}
John~M. {Lee}.
\newblock {\em Introduction to Smooth Manifolds}.
\newblock Springer, 2nd edition, 2013.

\bibitem{Maggi}
F.~{Maggi}.
\newblock {\em {Sets of Finite Perimeter and Geometric Variational Problems}}.
\newblock Cambridge University Press, 1st edition, 2012.

\bibitem{NevesMarques}
F.C. {Marques} and A.~{Neves}.
\newblock {Min-Max theory and the Willmore conjecture}.
\newblock {\em Ann. of Math.}, 149:683--782, 2014.

\bibitem{NdiayeSchaetzle}
C.~B. {Ndiaye} and R.~{Sch\"atzle}.
\newblock {A Convergence Theorem for Immersions with $L^2$-Bounded Second
  Fundamental Form.}
\newblock {\em Rend. Semin. Mat. Univ. Padova}, 127:235–248, 2012.

\bibitem{Nitsche}
J.C.C. {Nitsche}.
\newblock {Boundary value problems for variational integrals involving surface
  curvatures.}
\newblock {\em Quart. Appl. Math.}, 51:363--387, 1993.

\bibitem{OuYang}
Z.~{Ou-Yang}.
\newblock {Elasticity theory of biomembranes}.
\newblock {\em Thin Solid Films}, 393:19--23, 2001.

\bibitem{OuYang_Helfrich}
Z.~{Ou-Yang} and W.~{Helfrich}.
\newblock Bending energy of vesicle membranes: General expressions for the
  first, second, and third variation of the shape energy and applications to
  spheres and cylinders.
\newblock {\em Phys. Rev. A}, 39(10):5280--5288, 1989.

\bibitem{SchaetzleCurrentLower}
R.~{Sch\"atzle}.
\newblock {Lower semicontinuity of the Willmore functional for currents}.
\newblock {\em J. Differential Geom.}, 81(2):437--456, 2009.

\bibitem{Schaetzle}
R.~{Sch\"atzle}.
\newblock {The Willmore boundary problem.}
\newblock {\em Calc. Var. Partial Differ. Equ.}, 37:275--302, 2010.

\bibitem{Scholtes}
S.~{Scholtes}.
\newblock Elastic catenoids.
\newblock {\em Analysis}, 31:125--143, 2011.

\bibitem{Schygulla}
J.~{Schygulla}.
\newblock {Willmore Minimizers with Prescribed Isoperimetric Ratio}.
\newblock {\em Arch. Rational Mech. Anal.}, 203(3):901--941, 2012.

\bibitem{Schaetzle_Skript_Allard}
R.~{Schätzle}.
\newblock Geometrische {Ma\ss theorie}.
\newblock Lecture Notes, Tübingen University, 2008/09.

\bibitem{Simon_Buch}
L.~{Simon}.
\newblock {\em {Lectures on Geometric Measure Theory}}.
\newblock Proceedings of the Centre For Mathematical Analysis, Australian
  National University, 1st edition, 1983.

\bibitem{Simon}
L.~{Simon}.
\newblock { Existence of surfaces minimizing the Willmore functional.}
\newblock {\em Comm. Anal. Geom.}, 1:281--326, 1993.

\bibitem{Thomsen}
G.~{Thomsen}.
\newblock {\"Uber Konforme Geometrie I: Grundlagen der konformen
  Fl\"achentheorie}.
\newblock {\em Hamb. Math. Abh.}, 3:31--56, 1924.

\bibitem{Willmore}
T.J. {Willmore}.
\newblock {Note on embedded surfaces}.
\newblock {\em An. \c{S}tiin\c{t}. Univ. Al. I. Cuza Ia\c{s}i Se\c{c}t. I a
  Mat}, 11:493--496, 1965.

\end{thebibliography}
\bibliographystyle{plain}

\end{document}